\documentclass[10pt,a4paper,abstract=true,sfdefaults=false]{scrartcl}

\usepackage[T1]{fontenc}
\usepackage[utf8]{inputenc}
\usepackage[english]{babel}
\usepackage[autostyle]{csquotes}
\usepackage{lmodern}

\usepackage{amsmath, amsfonts, amssymb, amsthm, mathtools, esint}

\numberwithin{equation}{section}

\usepackage[a4paper, margin=1in]{geometry}
\usepackage{microtype}

\usepackage{graphicx}

\usepackage{tikz}
\usetikzlibrary{calc, intersections}
\tikzset{>=latex, shorten >=1pt}

\usepackage{pgfplots}
\usepgfplotslibrary{groupplots}
\pgfplotsset{compat=newest}

\usepackage{xcolor}
\usepackage[colorlinks=true, linkcolor=blue, citecolor=blue, urlcolor=blue]{hyperref}
\usepackage{url}
\usepackage{xurl}

\usepackage{enumitem}

\setkomafont{disposition}{\normalfont\bfseries}
\setkomafont{title}{\normalfont\Large\bfseries}
\setkomafont{subtitle}{\normalfont\large}
\setkomafont{author}{\normalfont}
\setkomafont{date}{\normalfont}
\setkomafont{publishers}{\normalfont\small}

\usepackage[backend=biber,style=alphabetic,sorting=ynt]{biblatex}
\appto\bibsetup{\sloppy\emergencystretch=1em}
\usepackage[automark,pagestyleset=standard,markcase=upper]{scrlayer-scrpage}
\usepackage{zref-clever} 
\zcsetup{cap}

\theoremstyle{definition}

\newtheorem{theorem}{Theorem}[section]

\newtheorem{definition}[theorem]{Definition}
\newtheorem{lemma}[theorem]{Lemma}
\newtheorem{remark}[theorem]{Remark}
\newtheorem{corollary}[theorem]{Corollary}
\newtheorem{example}[theorem]{Example}
\newtheorem{proposition}[theorem]{Proposition}

\makeatletter
\renewenvironment{proof}[1][Proof]{%
  \par\pushQED{\qed}%
  \normalfont
  \trivlist
  \item[\hskip\labelsep\bfseries #1\@addpunct{.}]%
}{%
  \popQED\endtrivlist\ignorespacesafterend
}
\renewcommand*\@pnumwidth{2em} 
\makeatother

\newcommand{\restrict}[1]{\left.\right|_{#1}}

\DeclareMathOperator{\dist}{dist}

\DeclareMathOperator{\const}{const}

\DeclareMathOperator{\divg}{div}

\newcommand{\R}{\mathbb R}

\title{The Parabolic Harnack Inequality on Weighted Riemannian Manifolds}
\author{Stefan Christian Kohlmeier}
\date{\today}
\publishers{}

\begin{document}

\maketitle

\begin{abstract}
    We establish the parabolic Harnack inequality on weighted Riemannian manifolds for a large class of parabolic differential operators building on an approach due to Alexander Grigor'yan.
\end{abstract}

\section{Introduction}\label{sec:introduction}

Let \((M, g, \mu)\) be a smooth, non-compact, connected and complete weighted Riemannian manifold, which may have a boundary. 

Our main problem is to establish a parabolic Harnack inequality on Riemannian manifolds for a general class of operators under minimal assumptions on the underlying space \((M,g,\mu\)).

We briefly review some historical background on the problem. Peter Li and Shing-Tung Yau proved in 1986, assuming $M$ has non-negative Ricci curvature and convex boundary, the famous gradient estimate
\begin{equation*}
	\frac{n}{2t} \geq \frac{|\nabla u|^2}{u^2} - \frac{\partial_t u}{u}
\end{equation*}
for solutions $u$ of the heat equation in $M$ with Neumann condition on $\partial M$. This estimate implies both the parabolic Harnack inequality and two-sided Gaussian bounds of the heat kernel; see \cite{liyau1986kernel,li2012geometric}. At this point of time, it was well known among experts that the parabolic Harnack inequality is equivalent to two-sided Gaussian heat kernel estimates via methods from \cite{fabesstroock1986harnack,liyau1986kernel}. Building on Landis' method, Alexander Grigor'yan showed in 1991 that on non-compact Riemannian manifolds the assumption of non-negative Ricci curvature can be replaced by the weaker requirement of volume doubling together with the (scale invariant) Poincar\'e inequality; see \cite{grigoryan1991heat}. Independently, Laurent Saloff-Coste proved the same result by adapting the method from \cite{moser1964harnack} in 1992; see \cite{saloff2002aspects,saloff1992note}. He also proved that volume doubling and the Poincar\'e inequality are necessary conditions for the parabolic Harnack inequality.

In his work, Grigor'yan also indicates that the argument extends to more general uniformly parabolic divergence form operators; see \cite{grigoryan1991heat}.

The purpose of this paper is to carry out this extension in detail in the setting of weighted Riemannian manifolds, possibly with boundary. This is a result obtained in my Master's thesis \cite{kohlmeier2025gaussian}. The thesis has not been published, but is available upon request. Another purpose of this paper is to give a detailed account of Landis' and Grigor'yan's methods of proving the parabolic Harnack inequality, since it does not seem to occur in the literature as often.

Let \(p \in C^\infty(M)\) be uniformly positive and bounded, and let \(U(x,t) : T_xM \to T_x M\) be a smooth symmetric uniformly elliptic family of operators. We consider the operator
\begin{equation*}
	Lu := p \partial_t u - \divg(U \nabla u)
\end{equation*}
and prove:

\begin{theorem}\label{thm:introduction_parabolic_harnack_inequality}
    Assume $M$ has the following two geometric properties:
	\begin{enumerate}[label=(\alph*)]
		\item\label{enum:introduction:1} There exists a real number $A > 0$ such that for arbitrary \(z \in M\) and \(R > 0\)
		\begin{equation*}
			V(z,2R) \leq A V(z,R), \quad \text{(Volume Doubling)}
		\end{equation*}
		where \(V(z,R) := \mu(B(z,R))\) is the volume function.
		\item\label{enum:introduction:2} There exists a real number \(a > 0\) such that for any \(z \in M\), \(R > 0\) and smooth function \(f \in C^\infty(B(z,R))\)
		\begin{equation*}
			\int_{B(z,R)} |\nabla f|^2d\mu \geq \frac{a}{R^2} \inf_{\xi \in \mathbb{R}} \int_{B(z,R)}(f-\xi)^2d\mu \quad \text{(Poincar\'e Inequality)}
		\end{equation*}
		is satisfied.
	\end{enumerate}
    Then, any non-negative solution \(u\) of the parabolic problem
    \begin{equation*}
	   \begin{cases}
		      Lu &= 0 \quad \text{in} \ B(z,2R) \times (0, 4R^2)\\
		      \langle U \nabla u, n\rangle &= 0 \quad \text{on} \ (\partial M \cap B(z,2R)) \times (0, 4R^2) 
	   \end{cases}
    \end{equation*}
    satisfies the parabolic Harnack inequality
    \begin{equation*}
	   \sup_{B(z,R) \times (R^2, 2R^2)} u \leq C \inf_{B(z,R) \times (3R^2, 4R^2)} u
    \end{equation*}
    for a constant \(C > 0\) depending only on the doubling and Poincar\'e constants as well as on the uniform bounds of \(p\) and \(U\).
\end{theorem}

As a consequence we obtain

\begin{theorem}\label{thm:introduction_parabolic_harnack_inequality_heat_equation_equivalent_condition}
    For such weighted Riemannian manifolds, the validity of the parabolic Harnack inequality for the heat equation is equivalent to its validity for all operators \(L\) in the class considered above.
\end{theorem}

\paragraph{Acknowledgements.} I would like to express my gratitude to Simon Nowak, the supervisor of my Master’s thesis, and to Alexander Grigor’yan for their useful discussions and support in completing this thesis.

\section{Prerequisites}\label{sec:prerequisites}

In this section, we briefly review the facts and definitions that are needed for the proof.

\subsection{Basic Notions}
\label{subsec:basic_notions}
Throughout the paper we write \(\const\) for a positive constant which may change from line to line, and \(\const_{x,y,z, \dots}\) for a positive constant that may depend on \(x,y,z,\dots\).

For two non-negative functions \(f,g\) we use the notation \(f \lesssim g\) if there exists an absolute constant \(C > 0\) such that the inequality
\begin{equation*}
	f \leq Cg
\end{equation*}
is satisfied. The relation \(f \gtrsim g\) is defined in the same manner. We write \(f \simeq g\) if both \(f \lesssim g\) and \(f \gtrsim g\) hold, i.e. there exist absolute constants \(C,c > 0\) such that
\begin{equation*}
	cg \leq f \leq Cg.
\end{equation*}

Let \((M, g, \mu)\) be a smooth weighted Riemannian manifold with or without boundary. The measure \(\mu\) is induced by a smooth, positive function \(D \in C^\infty(M)\) named density such that \(d\mu = Dd\nu_g\), where \(\nu_g\) is the canonical Riemannian measure. In local coordinates the central objects divergence \(\divg_g\) and the Laplace Beltrami operator \(\Delta_g\) are given by
\begin{equation*}
    \divg_g v = \frac{1}{\sqrt{\det g}} \frac{\partial}{\partial x^k} \left(\sqrt{\det g} v^k\right)
\end{equation*}
for all smooth vector fields \(v \in C^\infty(TM)\) and
\begin{equation*}
    \Delta_g f = \divg_g(\nabla_g f) = \frac{1}{\sqrt{\det g}} \frac{\partial}{\partial x^i} \left(\sqrt{\det g} g^{ij} \frac{\partial f}{\partial x^j}\right)
\end{equation*}
for all smooth functions \(f \in C^\infty(M)\). We have used the Einstein summation here. We also introduce the notion of the weighted divergence \(\divg_{g, \mu}\) and the weighted Laplace Beltrami operator \(\Delta_{g, \mu}\)
\begin{equation*}
    \divg_{g, \mu} v = \frac{1}{D} \divg_g Dv = \frac{1}{D \sqrt{\det g}} \frac{\partial}{\partial x^k} \left(D \sqrt{\det g}v^k\right)
\end{equation*}
for all smooth vector fields \(v \in C^\infty(TM)\) and
\begin{equation*}
    \Delta_{g, \mu} f = \divg_{g, \mu}(\nabla_g f) = \frac{1}{D \sqrt{\det g}} \frac{\partial}{\partial x^i} \left(D \sqrt{\det g} g^{ij} \frac{\partial f}{\partial x^j}\right)
\end{equation*}
for all smooth functions \(f \in C^\infty(M)\), respectively. Since we are mostly dealing with weighted manifold in the succeeding content, we will avoid lower indices for simplicity if it is clear what is meant (i.e. \(\nabla_g = \nabla\), \(\divg_{g, \mu} = \divg\) and \(\Delta_{g, \mu} = \Delta\)).

\subsection{The Divergence Theorem}\label{subsec:the_divergence_theorem}
For the sake of completeness, we state the Divergence Theorem for weighted Riemannian manifolds.

\begin{theorem}[Weighted Divergence Theorem]\label{thm:weighted_divergence_theorem}
    We have for any smooth and compactly supported vector field \(v \in C_0^\infty(TM)\) the identity
    \begin{equation*}
        \int_M \divg_{g, \mu} v d\mu = \int_{\partial M} \langle v, n\rangle_g d\tilde{\mu},
    \end{equation*}
    where \(n\) is the outward pointing unit normal vector field along \(\partial M\) and \(d\tilde{\mu} =  Dd\tilde{\nu}_g\) is the induced measure on \(\partial M\).
\end{theorem}
\begin{proof}
    Follows immediately from the ordinary Divergence Theorem.
\end{proof}

\begin{corollary}[Integration by Parts]\label{cor:integration_by_parts}
    In the setting of \zcref{thm:weighted_divergence_theorem} we have for any smooth function \(u \in C^\infty(M)\) the formula
    \begin{equation*}
        \int_M \langle \nabla u, v\rangle_g d\mu = \int_{\partial M} u \langle v, n\rangle_g d\tilde{\mu} - \int_M u \divg_{g, \mu} v d\mu.
    \end{equation*}
\end{corollary}

\subsection{Sobolev Spaces}\label{subsec:sobolev_spaces}
Let \((M,g, \mu)\) be a weighted manifold, \(\Omega \subset M\) a domain and let us denote by \(\vec{L}^2(\Omega)\) the space of all measurable vector fields \(v(x)\) on \(\Omega\) such that \(|v|_g \in L^2(\Omega)\). Analogously, we define \(\vec{L}_{loc}^2(\Omega)\) to be the space of all measurable vector fields \(v(x)\) on \(\Omega\) such that \(|v|_g \in L^2_{loc}(\Omega)\). The space \(\vec{L}^2(\Omega)\) is a Hilbert space with respect to the inner product
\begin{equation*}
	(v,w)_{\vec{L}^2(\Omega)} := \int_\Omega \langle v, w\rangle_g d\mu.
\end{equation*}\par

As usual in weak formulations, we denote by \(\mathcal{D}(\Omega)\) the space of smooth and compactly supported functions on \(\Omega\) and by \(\vec{\mathcal{D}}(\Omega)\) the space of smooth and compactly supported vector fields on \(\Omega\). Note that compactly supported functions in \(\Omega\)
  are allowed to be non-zero on \(\Omega \cap \partial M\).

\begin{definition}\label{def:weak_gradient}
	Given \(u \in L_{loc}^2(M)\), a \emph{weak gradient} of \(u\) is a vector field \(v = \nabla u \in \vec{L}^2_{loc}(M)\) such that
	\begin{equation*}
		\int_M u \divg \psi d\mu = - \int_M \langle v, \psi \rangle_g d\mu \quad \text{for all} \ \psi \in \vec{\mathcal{D}}(\mathring{M}).
	\end{equation*}
\end{definition}

Next, we define the Sobolev space \(W^{1,2}\) on \(M\). Note that it is also possible to define Sobolev spaces of higher order, but these are beyond our interest.

\begin{definition}\label{def:sobolev_space}
The \emph{Sobolev space} \(W^{1,2}(\Omega)\) is defined by
\begin{equation*}
	W^{1,2}(\Omega) := \left\{u \in L^2(\Omega) : \nabla u \in \vec{L}^2(\Omega) \right\},
\end{equation*}
and endowed with the inner product
\begin{equation*}
	(u,v)_{W^{1,2}} := (u,v)_{L^2} + (\nabla u,\nabla v)_{\vec{L}^2}.
\end{equation*}\par
We also use the following related spaces:
	\begin{itemize}
  		\item \(W^{1,2}_{loc}(\Omega)\): functions \(u\) such that \(u \in W^{1,2}(U)\) for all pre-compact domains \(U \subset \Omega\);
  		\item \(W^{1,2}_0(\Omega)\): the closure of \(\mathcal{D}(\Omega)\) in \(W^{1,2}(\Omega)\);
  		\item \(W^{1,2}_c(\Omega)\): the subspace of compactly supported functions in \(W^{1,2}(\Omega)\).
	\end{itemize}
\end{definition}

\begin{remark}\label{rem:meyers_serrin_on_manifolds}
	There is also the possibility to define \(W^{1,2}\) as the completion of smooth functions in \(L^2\) by the norm induced via the inner product mentioned in \zcref{def:sobolev_space}. A proof can be found in \cite{chan2024meyersserrintheoremriemannianmanifolds}. The argument relies on a local smooth approximation theorem for domains in \(\R^n\) with smooth boundary, as stated in \cite{adamsfournier2003sobolev,evans2010partial}. We will not use this characterization in what follows.
\end{remark}

\subsection{The Weak Laplace Equation}\label{sec:the_weak_Laplace_equation}

\begin{definition}\label{def:weak_laplacian}
	Let \(u \in W^{1,2}_{loc}(\Omega)\) and \(f \in L_{loc}^2(\Omega)\). We say that the equation
	\begin{equation*}
		\begin{cases}
			\Delta u &= f\\
			\langle \nabla u, n\rangle_g \restrict{\partial M \cap \Omega} &= 0
		\end{cases}
	\end{equation*}
	holds weakly in \(\Omega\) if
	\begin{equation*}
		\int_\Omega \langle \nabla u, \nabla \varphi\rangle_g d\mu = - \int_\Omega f\varphi d\mu \quad \text{for all} \ \varphi \in \mathcal{D}(\Omega).
	\end{equation*}
	\(n\) is the outward-pointing unit normal vector field along \(\partial M\).
\end{definition}

The preceding definition is motivated by the following observation: Let \(\Omega\) have a smooth boundary and consider a smooth function \(u \in C^\infty(\Omega)\). We have
\begin{equation*}
	\begin{split}
		\int_\Omega \varphi \Delta u d\mu &= -\int_{\Omega} \langle \nabla u, \nabla \varphi\rangle_g d\mu + \int_{\Omega \cap \partial M} \langle \nabla u, n\rangle \varphi d\widetilde{\mu} + \int_{\partial \Omega} \langle \nabla u, n\rangle \varphi d\widetilde{\mu}\\
		&= \int_{\Omega} f \varphi d\mu + \int_{\Omega \cap \partial M} \langle \nabla u, n\rangle \varphi d\widetilde{\mu}
	\end{split}
\end{equation*}
for all functions \(\varphi \in \mathcal{D}(\Omega)\). Therefore, if the weak formulation is valid, then \(\Delta u = f\) and \(\langle \nabla u, n \rangle \restrict{\Omega \cap \partial M} = 0\) follow.

We refer to \cite{chavel1984eigenvalues,davies1989heat,grigoryan2009heat,kohlmeier2025gaussian} for background on heat kernels, which will not be used in what follows.

\subsection{Known Results}\label{subsec:known_results}

In this subsection, we assume that \((M,g,\mu)\) is a smooth, non-compact, connected, and complete weighted Riemannian manifold, possibly with boundary, satisfying the volume doubling property as well as the Poincaré inequality stated in \ref{enum:introduction:2}. The latter may equivalently be formulated as follows; see \cite{jerison1986poincare}:

There exists a real number \(a > 0\) and a natural number \(N \geq 1\) such that for any \(z \in M\), \(R > 0\) and \(f \in C^\infty(B(z, NR))\) the inequality
\begin{equation*}
	\int_{B(z, NR)} |\nabla f|^2 d\mu \geq \frac{a}{R^2} \inf _{\xi \in \R} \int_{B(z,R)} (f-\xi)^2d\mu. \quad \text{(Weak Poincar\'e Ineq.)} 
\end{equation*}
is satisfied.

To keep the argument general, we keep track of the constant \(N\).

\begin{proposition}[{\cite[Theorem 1.1]{grigoryan1991heat}}]\label{prop:balls_ratio_estimate}
	We have for any \(x,y \in M\) and real numbers \(0 < r \leq R\) such that the intersection \(\overline{B(x,r)} \cap \overline{B(y,R)}\) is non-empty the estimate
	\begin{equation*}
		A_2 \left(\frac{R}{r}\right)^{a_2} \leq \frac{V(x, R)}{V(y, r)} \leq A_1 \left(\frac{R}{r}\right)^{a_1},
	\end{equation*}
	where \(A_i\) and \(a_i\) are positive real numbers only depending on \(A\) for \(i = 1,2\).
\end{proposition}

\begin{remark}\label{rem:balls_ratio_estimate}
	The upper bound in \zcref{prop:balls_ratio_estimate} is just obtained by manipulating the volume doubling property by algebraic manipulations whereas the lower bound follows from the special structure of the manifold.
\end{remark}

\begin{theorem}[Poincare-Type Inequality 1, {\cite[Theorem 1.2]{grigoryan1991heat}}]\label{thm:poincare_1}
	Let us fix \(z \in M\), \(R > 0\) and \(\varepsilon \in (0,1]\). For each function \(f \in W^{1,2}(B(z, (1+\varepsilon)R))\), we have the following estimates
	\begin{equation}\label{eq:thm:poincare_1:1}
		\int_{B(z, (1+\varepsilon)R)} |\nabla f_+|^2 d\mu \geq \const_{A,a,N} \frac{\varepsilon^\alpha \mu(H)}{R^2 V(z,R)} \int_{B(z,R)}f_+^2 d\mu
	\end{equation}
	and
	\begin{equation}\label{eq:thm:poincare_1:2}
		\int_{B(z, (1+\varepsilon)R)} |\nabla f|^2 d\mu \geq \const_{A,a,N} \frac{\varepsilon^\alpha}{R^2} \inf_{\xi \in \R} \int_{B(z,R)} (f-\xi)^2 d\mu.
	\end{equation}
	Here \(H = \{f \leq 0\} \cap B(z,R)\) and \(\alpha=\alpha(A) > 0\).
\end{theorem}

\begin{theorem}[Poincar\'e-Type Inequality 2, {\cite[Theorem 1.3]{grigoryan1991heat}}]\label{thm:poincare-type_inequality}
	Let \(x  \in M\), \(R > 0\) and \(f \in W^{1,2}(B(x,R))\) and set 
	\begin{equation*}
		H := \{y \in B(x, R/2) : f(y) \leq 0\}.
	\end{equation*}		
	Then 
	\begin{equation}\label{eq:thm:poincare-type_inequality:1}
		\int_{B(x,R)} |\nabla f_+|^2 \eta^2 d\mu \geq \const_{A,a,N} \frac{\mu(H)}{R^2 V(x,R)^2} \left( \int_{B(x,R)}f_+ \eta^2d\mu\right)^2.
	\end{equation} 
	Here, \(\eta(y) = (\dist(y, \partial B(x,R))/R)^{\alpha/2}\) and \(\alpha = \alpha(A) > 0\) is a constant.
\end{theorem}

\begin{theorem}[Bottom Eigenvalue Estimate, {\cite[Theorem 1.4]{grigoryan1991heat}}]\label{thm:bottom_eigenvalue_estimate}
	Let \(z \in M\), \(R > 0\) and fix a domain \(\Omega\) with \(\overline{\Omega} \subset B(z, R)\). Let us denote by \(\lambda_1 = \lambda_1(\Omega)\) the bottom eigenvalue of the \emph{Dirichlet problem}
	\begin{align*}
		\begin{cases}
			\Delta u + \lambda u &= 0\\
			u\restrict{\partial \Omega} &= 0\\
			\langle \nabla u, n\rangle \restrict{\partial M \cap \Omega} &= 0,
		\end{cases}
	\end{align*}
	where \(u \in C^\infty(\overline{\Omega})\) is the sought function and \(n\) is the outward-pointing unit normal vector field along \(\partial M\).\footnote{In the case \(\partial M = \emptyset\) this is precisely the Dirichlet problem and in the general case we refer to it as the Dirichlet problem as well.} Then the bottom eigenvalue is bounded below by
	\begin{equation*}
		\lambda_1 \geq \frac{b}{R^2} \left(\frac{V(z,R)}{\mu(\Omega)}\right)^\beta,
	\end{equation*}
	and \(b = b(A,a,N) > 0\) and \(\beta = \beta(A) > 0\) are two constants.
\end{theorem}

\section{\texorpdfstring{\(L^2\)}{L2}-Mean Value Inequality}\label{sec:l2-mean_value_inequality}

In this section, we establish a mean value inequality for the operator \(L\) introduced in \zcref{sec:introduction}, under the assumption of an isoperimetric inequality. This is an important step toward proving the parabolic Harnack inequality.

\subsection{Isoperimetric Inequality}\label{subsec:isoperimetric_inequality}

In this section we introduce the notion of an isoperimetric inequality.

\begin{definition}\label{def:isoperimetric_inequality}
	Let \(\Lambda : (0, \infty) \to \R\) be a continuous, positive and monotone decreasing function. We say that an \emph{isoperimetric inequality} is satisfied in a domain \(\Omega \subset M\) if for any open set \(U \subset M\) with \(\overline{U} \subset \Omega\) we have
	\begin{equation*}
		\lambda_1(U) \geq \Lambda(\mu(U))
	\end{equation*}
	for the bottom eigenvalue \(\lambda_1 = \lambda_1(U)\) of the Dirichlet problem in \(U\) introduced in \zcref{thm:bottom_eigenvalue_estimate}.
\end{definition}

\begin{example}\label{ex:isoperimetric_inequality_cartan_hardamard}
	If \(M\) is a Cartan-Hadamard manifold, that is, a simply connected Riemannian manifold with non-positive sectional curvature, there is an isoperimetric inequality with the function
	\begin{equation*}
		\Lambda(v) = av^{-2/n},
	\end{equation*}
	where \(a = a(n) > 0\). Some typical examples are  \(\R^n\) and \(\mathbb{H}^n\). Moreover, it holds for minimal submanifolds of \(\R^n\). This example is taken from \cite{grigoryan2009heat}; for more details see e.g. \cite{chunggrigoryanyau1999}.
\end{example}

\begin{example}\label{ex:isoperimetric_inequality_volume_doubling_poincare}
	Let us assume that \(M\) admits the volume doubling property and the weak Poincar\'e inequality. \zcref[S]{thm:bottom_eigenvalue_estimate} implies that in every ball \(B(z, R)\), where \(z \in M\) and \(R > 0\), there is an isoperimetric inequality with the function	
	\begin{equation}
		\Lambda(v) = \frac{b}{R^2} \left(\frac{V(z,R)}{v}\right)^\beta,
	\end{equation}
	where \(b = b(A,a,N) > 0\) and \(\beta = \beta(A) > 0\) are two positive constants.
\end{example}

We note that the function
\begin{equation*}
	\begin{split}
		\tau : (0, \infty) &\to (0, \infty)\\
		v &\mapsto v/\Lambda(v)
	\end{split}
\end{equation*}
is strictly monotonically increasing and continuous. Furthermore, we have \(\tau(v) \to \infty\) as \(v \to \infty\). Therefore, it admits an inverse function
\begin{equation*}\label{eq:sec:isoperimetric_inequality:2}
	\omega : (0, \infty) \to (0,\infty),
\end{equation*}
which is continuous and strictly monotonically increasing as well. For any \(t, r > 0\) and a given absolute constant \(c > 0\) let us define functions \(V(t), W(r)\) by the equalities
\begin{equation}\label{eq:sec:isoperimetric_inequality:3}
	ct = \int_0^{V(t)} \frac{d\xi}{\omega(\xi)} \quad \text{and} \quad cr = \int_0^{W(r)} \frac{d\xi}{\sqrt{\xi\omega(\xi)}}.
\end{equation}

\begin{remark}
	Note that we have not specified any domain or codomain for the functions \(V\) and \(W\). These are given implicitly. If there exists a value \(a \geq 0\) such that
	\begin{equation*}
		ct = \int_0^a \frac{d\xi}{\omega(\xi)},
	\end{equation*}
	then we set \(V(t) = a\). Otherwise, \(V(t)\) remains undefined. However, if \(V(t)\) is defined, then \(V(t')\) is also defined for all \(0 < t' < t\) due to the continuity of the integral. The function \(W\) is treated analogously.
\end{remark}

In what follows, we assume that the integrals in \zcref{eq:sec:isoperimetric_inequality:3} converge in a neighbourhood of zero. Hence, for sufficiently small \(R\) and \(T\), the functions \(V\) and \(W\) defined in \zcref{eq:sec:isoperimetric_inequality:3} are well-defined in \((0,T)\) and \((0,R)\), respectively.

\begin{example}\label{ex:V_W_polynomial_isoperimetric_inequality}
	 In the situation of \zcref{ex:isoperimetric_inequality_cartan_hardamard,ex:isoperimetric_inequality_volume_doubling_poincare} this takes place. Indeed, let us consider the more general case
	 \begin{equation*}
	 	\Lambda(v) \gtrsim v^{-\nu}
	 \end{equation*}
	 in a neighbourhood of zero, where \(\nu > 0\) is a given constant. We get \(\tau(v) \lesssim v^{1+\nu}\), so
	 \begin{equation*}
	 	\omega(\xi) \gtrsim \xi^{1/(1+\nu)}
	 \end{equation*}
	 holds in a neighbourhood of zero. We obtain for \(a > 0\) small enough
	 \begin{equation*}
	 	\int_{0+} \frac{d\xi}{\omega(\xi)} \lesssim \int_{0+} \xi^{-1 /(1+\nu)} d\xi < \infty,
	 \end{equation*}
	 and in the same way
	 \begin{equation*}
	 	\int_{0+} \frac{d\xi}{\sqrt{\xi \omega(\xi)}} \lesssim \int_{0+} \xi^{-(2+\nu)/(2+2\nu)} d\xi < \infty.
	 \end{equation*}
     This example is essentially contained in \cite{grigoryan1991heat}.
\end{example}

\subsection{\texorpdfstring{\(L^2\)}{L2}-Mean Value Inequality}\label{subsec:l2-mean_value_inequality}

Let us consider the operator  \(L : C^\infty(M \times (0, \infty) ) \to  C^\infty(M \times (0, \infty))\) defined by
\begin{equation*}
		Lu := p \partial_t u - \divg(U \nabla u)
\end{equation*}
for a fixed smooth, uniformly positive and bounded function \(p \in C^\infty(M)\) and a smooth family of uniformly elliptic operators \(U(x,t): T_xM \to T_x M\).  A smooth symmetric family of operators \(U(x,t) : T_x M \to T_x M\) is called \emph{uniformly elliptic} if there exists \(\lambda > 0\) such that for any \(x \in M\) and \(\xi \in T_x M\) the following relation holds:
\begin{equation*}
	\lambda^{-1} |\xi|^2 \leq \langle U(x,t) \xi, \xi\rangle \leq \lambda |\xi|^2.
\end{equation*}
We agree on the convention that when we say a constant depends on \(U\) or \(p\), we actually mean that it depends on the uniform ellipticity of \(U\) and on the upper and lower bounds of \(p\), respectively.

\begin{remark}
	Given an open set \(U \subset M \times (0,\infty)\), the restriction \(L\restrict{U}\) is again a self-map, that is,
	\begin{equation*}
		L\restrict{U} : C^\infty(U) \to C^\infty(U).
	\end{equation*}
	Thus, \(L\) restricts naturally to functions defined on smaller domains.
\end{remark}
Due to uniform ellipticity we have the following positive definite and symmetric bilinear form
\begin{align*}
		\langle -, -\rangle_U : T_x M \times T_x M &\to \R\\
		(\zeta, \xi) &\mapsto \langle U \zeta, \xi\rangle.
\end{align*}
We denote by
\begin{align*}
	\|-\|_U : T_x M &\to \R\\
	\xi &\mapsto \sqrt{\langle\xi, \xi\rangle_U}
\end{align*}
the norm induced by \(U\).

Let us fix a point \(z \in M\) and two real numbers \(R\), \(T > 0\). We write
\begin{equation*}
	\amalg = B(z,R) \times (0,T)
\end{equation*}
for the cylinder centred at \(z\) with radius \(R\) and height \(T\). Moreover, we define the measure \(\overline{\mu}\) on \(M \times \R\) by \(d\overline{\mu} := d\mu dt\).

\begin{definition}\label{def:parabolic_subparabolic_superparabolic}
	Let \(u \in C^\infty(\overline{\amalg})\). We say that \(u\) is parabolic in \(\amalg\) with respect to \(L\) if it solves the Dirichlet problem
	\begin{align}\label{eq:def:parabolic_subparabolic_superparabolic:1}
		\begin{cases}
			Lu = 0\\
			\langle \nabla u, n\rangle_U  \restrict{\partial M \cap B(z,R)} &= 0,
		\end{cases}
	\end{align}
	where \(n\) denotes the outward pointing unit normal vector field along \(\partial M\). If the condition \(Lu = 0\) in \zcref{eq:def:parabolic_subparabolic_superparabolic:1} is replaced by \(Lu \leq 0\), we call \(u\) \emph{subparabolic}. Similarly, if \(Lu = 0\) is replaced by \(Lu \geq 0\) in \zcref{eq:def:parabolic_subparabolic_superparabolic:1}, we call \(u\) \emph{superparabolic}.
\end{definition}

\begin{theorem}[Energy Estimate]\label{thm:energy_estimate}
	Let \(\eta: M \times [0, \infty) \to \R\) be a Lipschitz function with \(\eta(-,0) = 0\) supported in \(\overline{B(z,R)}\) for all \(t \geq 0\). Fix a subparabolic function \(u\) with respect to \(L\) in \(\amalg\). Then the estimate
	\begin{equation}\label{eq:thm:energy_estimate:1}
			\int_{B(z,R)} p v^2(\cdot,T) \eta^2(\cdot,T) d\mu + \frac{1}{2} \int_\amalg \|\nabla (v\eta)\|_U^2d\overline{\mu} \leq 5\int_\amalg v^2 \left(\|\nabla \eta\|_U^2 + p |\eta \partial_t\eta|\right) d\overline{\mu}
	\end{equation}
	holds for every function
	\begin{equation*}
		v := (u-\theta)_+,
	\end{equation*}
	where \(\theta \geq 0\) is an arbitrary real number.
\end{theorem}
\begin{proof}
	The proof is split into three steps.\par
	\emph{Step 1:} We show that for any function \(\varphi \in C_0^\infty(B(z,R))\) and \(t \in (0,T)\)
	\begin{equation}\label{eq:thm:energy_estimate:2}
		\int_{B(z,R)} pv \partial_t v \varphi^2 d\mu \leq -\int_{B(z,R)} \langle \nabla(v\varphi^2), \nabla v\rangle_U d\mu
	\end{equation}
	holds. Let us fix a function \(\varphi \in C_0^\infty(B(z,R))\) and assume that \(\theta\) is a regular value of \(u\) and \(u\restrict{\partial M}\). This ensures that \(S =\{u \geq \theta\}\) is a manifold with piecewise smooth boundary \(\partial S = \{u = \theta\} \cup \left(\partial M \cap B(z,R)\right)\) and justifies the application of \zcref{cor:integration_by_parts} in the following calculation:
	\begin{equation*}
		\begin{split}
			\int_{B(z,R)} pv\partial_t v \varphi^2 d\mu &= \int_S pv \partial_t u \varphi^2 d\mu \leq \int_S \varphi^2 v \divg (U \nabla u) d\mu\\
			&= \int_{\partial S} \langle\nabla u, n\rangle_U v\varphi^2 d\widetilde{\mu} - \int_{B(z,R)} \langle \nabla u, \nabla(v\varphi^2)\rangle_U d\mu,
		\end{split}
	\end{equation*}
	where \(n\) is the outward pointing unit normal vector field along \(\partial S\) and \(\widetilde{\mu}\) the induced measure on \(\partial S\). We need to show that the integral over \(\partial S\) is zero.	Indeed, we have \(v = 0\) on the level set \(\{u = \theta\}\) and \(\langle \nabla u, n \rangle_U = 0\) on \(\partial M \cap B(z,R)\). Let us consider the case when \(\theta\) is a critical value of \(u\). We find a decreasing sequence of regular values \(\theta_k \in (0, \infty)\) with \(\lim_{k\to\infty} \theta_k = \theta\) since regular values are dense, which follows from Sard's theorem; see e.g. \cite{lee2012introduction}. We know by results from \cite{grigoryan2009heat}
	\begin{equation*}
		v_k := (u-\theta_k)_+ \xrightarrow{W^{1,2}} v \ \text{as} \ k \to \infty.
	\end{equation*}
	The Dominated Convergence Theorem yields
	\begin{equation*}
		\lim_{k \to \infty} \int_{B(z, R)} pv_k \partial_t v_k \varphi^2 d\mu = \int_{B(z,R)} pv\partial_t v \varphi^2 d\mu,
	\end{equation*}
	because
	\begin{equation*}
		|p v_k \partial_t v_k \varphi^2| \leq  |p u \partial_t u \varphi^2|
		 \in L^1(B(z,R)).
	\end{equation*}
	For the integral on the right-hand side we proceed similarly:
	\begin{equation*}
			|\langle \nabla u, \nabla(v_k\varphi^2)\rangle_U| \leq \const |\nabla u| (\varphi^2 |\nabla u| + 2v |\varphi||\nabla \varphi|) \in L^1(B(z,R)),
	\end{equation*}
	thereby
	\begin{equation*}
		\lim_{k \to \infty} \int_{B(z,R)} \langle \nabla u, \nabla(v_k\varphi^2)\rangle_U d\mu = \int_{B(z,R)} \langle \nabla u, \nabla(v\varphi^2)\rangle_U d\mu
	\end{equation*}
	follows by the Dominated Convergence Theorem.\par
	\emph{Step 2:} In this step we generalize \zcref{eq:thm:energy_estimate:2} to Lipschitz functions supported in \(\overline{B(z,R)}\). Let \(\varphi\) be such a function. We fix a real number \(r < R\) and a smooth cutoff function \(\eta\) of \(B(z,r)\) in \(B(z,R)\) such that \(\eta \varphi \in W_c^{1,2}(B(z,R)) \subset W_0^{1,2}(B(z,R))\) and rename \(\eta \varphi\) back into \(\varphi\). Consequently, we find a sequence of smooth functions \(\varphi_k \in C_0^\infty(B(z,R))\) with
	\begin{equation*}
		\varphi_k \xrightarrow{W^{1,2}} \varphi \ \text{as} \ k \to \infty.
	\end{equation*}
	Note
	\begin{equation*}
		\begin{split}
			\|\varphi^2 - \varphi_k^2\|_{L^1(B(z,r))} &= \|(\varphi + \varphi_k)(\varphi - \varphi_k)\|_{L^1(B(z,r))}\\
			&\leq \|\varphi + \varphi_k\|_{L^2(B(z,r))} \|\varphi - \varphi_k\|_{L^2(B(z,r))}\\
			&\xrightarrow{k \to \infty}0
		\end{split}		
	\end{equation*}
	and
	\begin{equation*}
		\begin{split}
			\|\nabla(v\varphi^2) &- \nabla(v\varphi_k^2)\|_{L^1(B(z,r))} = \|\nabla v (\varphi^2 - \varphi_k^2) + 2v(\varphi \nabla \varphi - \varphi_k \nabla \varphi_k)\|_{L^1(B(z,r))}\\
            &\leq \const (\|\varphi^2 - \varphi_k^2\|_{L^1(B(z,r))} + \|\varphi \nabla \varphi - \varphi_k \nabla \varphi_k\|_{L^1(B(z,r))})\\
            & \leq \const (\|\varphi^2 - \varphi_k^2\|_{L^1} + \|\varphi - \varphi_k\|_{L^2}\|\nabla \varphi\|_{L^2} + \|\varphi_k\|_{L^2}\|\nabla \varphi - \nabla \varphi_k\|_{L^2})\\
            &\xrightarrow{k \to \infty} 0.
		\end{split}
	\end{equation*}
	With these identities in hand, it follows that
	\begin{equation*}
		\left|\int_{B(z,r)} pv \partial_t v \varphi^2 d\mu - \int_{B(z,r)} pv \partial_t v \varphi_k^2 d\mu\right| \leq \const \int_{B(z,r)} |\varphi^2 - \varphi_k^2| d\mu \xrightarrow{k \to \infty} 0
	\end{equation*}		
	and
	\begin{equation*}
		\begin{split}
			\left|\int_{B(z,r)} \langle \nabla u, \nabla(v \varphi^2)\rangle_U d\mu - \int_{B(z,r)} \langle \nabla u, \nabla(v \varphi_k^2)\rangle_U d\mu\right| &\leq \const\int_{B(z,r)} |\nabla (v \varphi^2) - \nabla (v \varphi_k^2)| d\mu\\
			&\xrightarrow{k\to\infty} 0.
		\end{split}
	\end{equation*}
	 The previous work leads to the estimate
	\begin{equation*}
		\int_{B(z,r)} pv\partial_t v \varphi^2 d\mu \leq -\int_{B(z,r)} \langle \nabla u, \nabla(v \varphi^2)\rangle_U d\mu,
	\end{equation*}
	and the claim follows by letting \(r \to R\) due to the Dominated Convergence Theorem.\par
	\emph{Step 3:} We show \zcref{eq:thm:energy_estimate:1}. By what we have shown in Step 2, we can plug \(\varphi_t(x) = \eta(x,t)\) for all  \(t \geq 0\) into \zcref{eq:thm:energy_estimate:2}. Integrating this inequality in \(t\) yields
	\begin{equation*}
			\int_\amalg p v \partial_t v \eta^2 d\overline{\mu} \leq - \int_\amalg \langle \nabla(v \eta^2), \nabla v\rangle_U d\overline{\mu} = -\int_\amalg \|\nabla v\|_U^2 \eta^2 d\overline{\mu} -2\int_\amalg \langle \nabla v, \nabla \eta\rangle_U v \eta d\overline{\mu}.
	\end{equation*}
	Furthermore, performing a partial integration gives us
	\begin{equation}\label{eq:thm:energy_estimate:3}
		\begin{split}
			\int_\amalg p v \partial_t v \eta^2 d\overline{\mu} &= \frac{1}{2}\int_\amalg p \partial_t(v^2) \eta^2d\overline{\mu} = \frac{1}{2} \int_{B(z,R)}p \left( \int_0^T \partial_t (v^2) \eta^2 dt \right)d\mu\\
			&= \frac{1}{2}\int_{B(z,R)} p [v^2\eta^2]_{t=0}^{t=T} d\mu - \int_{B(z,R)} \int_0^T p v^2 \eta \partial_t \eta dt d\mu\\
			&= \frac{1}{2}\int_{B(z,R)} p v^2(\cdot, T) \eta^2(\cdot,T) d\mu - \int_\amalg p v^2 \eta \partial_t \eta  d\overline{\mu}.
		\end{split}
	\end{equation}
	To estimate the right-hand side, we notice that \(0 \leq \|\nabla v \eta + 2 \nabla \eta v\|_U^2\) implies
	\begin{equation*}
		-2 (\nabla v, \nabla \eta)_U v \eta \leq \frac{1}{2} \|\nabla v\|_U^2 \eta ^2 + 2 v^2 \|\nabla \eta\|_U^2,
	\end{equation*}
	and \(0 \leq \| \nabla v \eta - \nabla \eta v\|_U^2\) similarly leads to\footnote{We can also apply Young's inequality as we will see in the proof of \zcref{thm:weak_harnack_inequality}.}
	\begin{equation*}
		-\|\nabla v\|_U^2 \eta^2 \leq \|\nabla \eta\|_U^2 v^2 - \frac{1}{2} \|\nabla (v \eta)\|_U^2.
	\end{equation*}
	Hence, we obtain
	\begin{equation}\label{eq:thm:energy_estimate:4}
		\int_\amalg p v \partial_t v \eta^2 d\overline{\mu} \leq - \frac{1}{4} \int_\amalg \|\nabla (v \eta)\|_U^2 d\overline{\mu} + \frac{5}{2}\int_\amalg \|\nabla \eta\|_U^2 v^2 d\overline{\mu}.
	\end{equation}
	The desired result \zcref{eq:thm:energy_estimate:1} follows from \zcref{eq:thm:energy_estimate:3} and \zcref{eq:thm:energy_estimate:4}.
\end{proof}

\begin{lemma}\label{lem:comparison_integrated_cutoffs}
	Let \(u\) be a subparabolic function in the cylinder \(\amalg\) with respect to \(L\). Moreover, let us assume that in \(B(z,R)\) there is an isoperimetric inequality with function \(\Lambda(v)\). We consider the subcylinder \(\widetilde{\amalg}\) defined by
	\begin{equation*}
		\widetilde{\amalg} := B(z,R') \times (T', T),
	\end{equation*}		
	where \(0 < T' < T\) and \(0 < R' < R\) are two real numbers. Given \(\theta > 0\) we set
	\begin{equation*}
		H = \int_\amalg u_+^2 d\overline{\mu} \quad \text{and} \quad \widetilde{H} = \int_{\widetilde{\amalg}} (u-\theta)_+^2 d\overline{\mu},
	\end{equation*}
	and \(H\) and \(\widetilde{H}\) fulfill the inequality
	\begin{equation}\label{eq:lem:comparison_integrated_cutoffs:1}
		\widetilde{H} \leq \frac{C H}{\delta \Lambda(C\delta^{-1}\theta^{-2}H)}.
	\end{equation}
	Here, \(\delta = \min \{T', (R-R')^2\}\) and \(C\) denotes an absolute constant only depending on \(p\) and \(U\). In the case of the heat equation, that is, \(L u = \partial_t u - \Delta u\), we can choose \(C = 50\).
\end{lemma}
\begin{proof}
	\zcref{eq:lem:comparison_integrated_cutoffs:1} follows from \zcref{thm:energy_estimate} by a careful choice of \(\eta\). At first, we set \(\eta(x,t) = \eta_1(x) \eta_2(t)\) for two functions \(\eta_1\) \(\eta_2\). We define \(\eta_1\) by
	\begin{equation*}
		\eta_1(x) = \begin{cases}
			1, \quad & r < \frac{R+R'}{2}\\
			\frac{R- r}{R - \frac{R+R'}{2}} = 2 \frac{R-r}{R-R'}, \quad & \frac{R+R'}{2} \leq r < R\\
			0, \quad & r \geq R,
		\end{cases}
	\end{equation*}
	where \(r = d(z,x)\). The second line just means that \(\eta_1\) is defined to be linear along the radius in \(B(z,R) \setminus B\left(z, \frac{R+R'}{2}\right)\). We define the function \(\eta_2\) by the relation
	\begin{equation*}
		\eta_2(t) = \begin{cases}
			1, \quad & t \geq T'\\
			\frac{t}{T'}, \quad & t < T'.
		\end{cases}
	\end{equation*}
	Now we apply \zcref{thm:energy_estimate} for \(v = u_+\) and \(\tau \in [T', T]\). We obtain
	\begin{equation*}
		\|\nabla \eta\|_U^2 \leq \|\nabla \eta_1\|_U^2 \leq \lambda |\nabla \eta_1|^2 \leq \frac{4\lambda}{(R-R')^2} \leq \frac{\const_U}{\delta},
	\end{equation*}
	where \(\lambda\) is the ellipticity constant of \(U\) and
	\begin{equation*}
		p |\eta \partial_t \eta| \leq p |\partial_t \eta_2| \leq \frac{p}{T'} \leq \frac{\const_p}{\delta},
	\end{equation*}
	hence
	\begin{equation*}
		\int_{B(z, (R+R')/2)} pu(\cdot, \tau)_+^2 d\mu \leq \const_p \int_\amalg u_+^2 \left(\|\nabla \eta\|_U^2 + |\eta \partial_t \eta| \right)d\overline{\mu} \leq \frac{\const_{p, U}}{\delta} H
	\end{equation*}
	follows. We deduce from the uniform positivity of \(p\) that
	\begin{equation}\label{eq:lem:comparison_integrated_cutoffs:2}
		\int_{B(z, (R+R')/2)} u(\cdot, \tau)_+^2 d\mu \leq \frac{\const_{p, U}}{\delta} H.
	\end{equation}	
	Now let us define \(\eta_1\) by 
	\begin{equation*}
		\eta_1(x) = \begin{cases}
			1, \quad & r < R' \\
			\frac{\frac{R+R'}{2}-r}{\frac{R+R'}{2}-R'} = 1 - 2\frac{r-R'}{R-R'}, \quad & R' \leq r < \frac{R+R'}{2}\\
			0, \quad & r \geq \frac{R+R'}{2},
		\end{cases}
	\end{equation*}		
	where \(r = d(z,x)\). The function \(\eta_2\) remains unchanged. We obtain by applying \zcref{thm:energy_estimate} to the function \(v = (u-\theta)_+\) the estimate
	\begin{equation}\label{eq:lem:comparison_integrated_cutoffs:3}
		\int_\amalg |\nabla (v \eta)|^2 d\overline{\mu} \leq \frac{\const_{p, U}}{\delta} \int_\amalg v^2 d\overline{\mu},
	\end{equation}
	since similarly as above
	\begin{equation*}
		\|\nabla \eta\|_U^2 \leq \frac{4\lambda}{(R-R')^2} \leq \frac{\const_U}{\delta}
	\end{equation*}
	holds. For each fixed \(t\) the function \(v \eta\) has support in \(\overline{D_t}\), where
	\begin{equation*}
		D_t := \left\{x \in B\left(z, \frac{R+R'}{2}\right) : u(x,t) > \theta\right\}.
	\end{equation*}
	According to the variational property of the first eigenvalue \(\lambda_1\) of the Dirichlet problem mentioned in the proof of \zcref{thm:bottom_eigenvalue_estimate}, the estimate
	\begin{equation}\label{eq:lem:comparison_integrated_cutoffs:4}
		\int_{B(z,R)} |\nabla (v\eta)|^2 d\mu \geq \lambda_1(D_t) \int_{B(z,R)} (v\eta)^2 d\mu
	\end{equation}
	holds. An application of Markov's inequality to \(\mu(D_t)\) and estimating the result by \zcref{eq:lem:comparison_integrated_cutoffs:2} leads us to the relation
	\begin{equation}\label{eq:lem:comparison_integrated_cutoffs:5}
		\mu(D_t) \leq \frac{1}{\theta^2}\int_{B(z, (R+R')/2)} u(x,t)_+^2d\mu(x) \leq \const_{p, U}\frac{H}{\theta^2 \delta}
	\end{equation}
	for all \(\tau \in [T', T]\). Let us fix the constant \(C := \const_{p, U}\) and use the isoperimetric inequality together with \zcref{eq:lem:comparison_integrated_cutoffs:3,eq:lem:comparison_integrated_cutoffs:4,eq:lem:comparison_integrated_cutoffs:5} to deduce
	\begin{equation*}
		\begin{split}
			\frac{C}{\delta}\int_\amalg v^2 d\overline{\mu} &\geq \int_\amalg |\nabla (v\eta)|^2d\overline{\mu} = \int_{T'}^T \int_{B(z,R)} |\nabla (v\eta)|^2 d\overline{\mu} \geq \int_{T'}^T\int_{B(z,R)} \lambda_1(D_t) (v\eta)^2d\mu dt\\
			&\geq \Lambda(CH\theta^{-2}\delta^{-1})\int_{T'}^T \int_{B(z,R)} (v\eta)^2 d\mu dt \geq \Lambda(CH\theta^{-2}\delta^{-1}) \widetilde{H}.
		\end{split}
	\end{equation*}
	\zcref[S]{eq:lem:comparison_integrated_cutoffs:1} follows by rearranging terms. The second statement about the heat equation follows by an analysis of the constants in this special case.
\end{proof}

\begin{theorem}[\(L^2\)-Mean Value Inequality]\label{thm:mean_value_inequality}
	Let \(u \in C^\infty(\overline{\amalg})\) be a subparabolic function with respect to \(L\). If there is an isoperimetric inequality \(\Lambda(v)\) in the ball \(B(z,R)\), then  the estimate
	\begin{equation}\label{eq:thm:mean_value_inequality:1}
		u(z,T)_+^2 \leq \frac{4}{\min \{V(T), W(R)\}} \int_\amalg u_+^2 d\overline{\mu}.	
	\end{equation}
	holds. The constant \(c\) in \zcref{eq:sec:isoperimetric_inequality:3} is chosen sufficiently small, depending only on \(p\) and \(U\). In the special case  \(Lu = \partial_t u - \Delta u\) we may pick any absolute constant \( 0 < c \leq 1/1600\).
\end{theorem}
\begin{remark}
	\(T > 0\) and \(R > 0\) in the preceding \zcref{thm:mean_value_inequality} have to be chosen small enough such that \(V(T)\) and \(W(R)\) are well-defined.
\end{remark}
\begin{proof}
	The idea of the proof is to iterate \zcref{lem:comparison_integrated_cutoffs} for a well chosen value \(\theta > 0\). Consider a sequence of cylinders \(\amalg_k = B(z, r_k) \times (t_k, T)\), where
	\begin{equation}\label{eq:thm:mean_value_inequality:2}
		0 = t_0 < t_1 < t_2 < \dots \leq \frac{T}{2}, \quad R = r_0 > r_1 > r_2 > \dots \geq \frac{R}{2}
	\end{equation}
	and \((r_k - r_{k+1})^2 = t_{k+1} - t_k =: \delta_k\). Let \(\theta > 0\) be an arbitrary but fixed value to be chosen later. Furthermore, we set \(\theta_k = (2-2^{-k})\theta\) and
	\begin{equation*}
		H_k = \int_{\amalg_k} (u-\theta_k)^2_+ d\overline{\mu} \quad \text{and} \quad H = \int_{\amalg} u_+^2 d\overline{\mu}.
	\end{equation*}
	 We note \(H_{k+1} \leq H_k\). Our goal is to find an appropriate \(\theta\) such that \(H_k \to 0\) as \(k \to \infty\). This would imply
	\begin{equation*}
		\int_{B(z, R/2)} \int_{T/2}^T (u-2\theta)_+^2 dt d\mu = 0,
	\end{equation*}
	so we obtain \(u(z, T) \leq 2 \theta\). That is, \(u(z,T)_+^2 \leq 4\theta^2\), which is precisely \zcref{eq:thm:mean_value_inequality:1} for a suitable \(\theta > 0\).\par
	\begin{figure}
		\centering
		\begin{tikzpicture}
  			\fill[blue!30] (0,-1.5) arc(270:450:0.225 and 1.5)
   				-- (-3,1.5) arc(90:270:0.225 and 1.5)
    			-- cycle;
  			\fill[blue!20] (-3,0) ellipse (0.225 and 1.5);
  			\fill[blue!20] ( 0,0) ellipse (0.225 and 1.5);

  			\draw (0,-3) -- (-6,-3) node[pos=0.8, above] { \(\amalg_0\)};
 			\draw (0, 3) -- (-6, 3);
  			\draw (-6,0) ellipse (0.375 and 3);
  			\draw (0,-3) arc(270:450:0.375 and 3);
  			\draw[dashed] (0,3) arc(90:270:0.375 and 3);

  			\draw (0,-2.25) -- (-4.5,-2.25) node[pos=0.8, above] { \(\amalg_k\)};
  			\draw (0, 2.25) -- (-4.5, 2.25);
  			\draw (-4.5,0) ellipse (0.3 and 2.25);
  			\draw (0,-2.25) arc(270:450:0.3 and 2.25);
  			\draw[dashed] (0,2.25) arc(90:270:0.3 and 2.25);

  			\draw (0,-1.5) -- (-3,-1.5);
  			\draw (0, 1.5) -- (-3, 1.5) node[pos=0.5, below, inner sep=1pt, blue!70!black] { \(T/2\)};
  			\draw (-3,0) ellipse (0.225 and 1.5);
  			\draw[dashed] (0,1.5) arc(90:270:0.225 and 1.5);
  			\draw (0,-1.5) arc(270:450:0.225 and 1.5);

  			\draw[->, blue!70!black] (0.75,0.5) -- (0,0) node[pos=0, above right, inner sep=1pt] { \(B(z, R/2)\)};
  			\draw[->] (-6.75,0.5) -- (-6,0) node[pos=0, above left, inner sep=1pt] { \(B(z, R)\)};

  			\draw[->] (-6,-3.5) -- (0.5,-3.5) node[right] {\(t\)};
  			\draw (-6,-3.4) -- (-6,-3.6) node[midway, left] {\(0\)};
		\end{tikzpicture}
		\caption{Nested cylinders \(\amalg_k\)}
	\end{figure}
	\emph{Step 1:} In this step we construct the sequence \(\{\delta_k\}_{k=0}^\infty\). We apply \zcref{lem:comparison_integrated_cutoffs} for the cylinders \(\amalg_{k+1} \subset \amalg_k\) and the functions \(u - \theta_k\) and \(u - \theta_{k+1} = (u - \theta_k) - 2^{-(k+1)}\theta\). It follows
	\begin{equation}\label{eq:thm:mean_value_inequality:3}
		H_{k+1} \leq \frac{C H_k}{\delta_k \Lambda(C \delta_k^{-1} 4^{k+1} \theta^{-2}H_k)},
	\end{equation}
	where \(C\) is an absolute constant depending on \(U\) and \(p\). We show that we can arrange it to choose \(\theta\) and \(\{\delta_k\}_{k=0}^\infty\) such that for all \(k \geq 0\)
	\begin{equation}\label{eq:thm:mean_value_inequality:4}
		H_k \leq \frac{H}{16^k}.
	\end{equation}
	The proof goes by induction on \(k\). For \(k = 0\), the inequality is trivial. Suppose for \(k = m\) we have already found \(\delta_1, \dots, \delta_{m-1}\) such that \zcref{eq:thm:mean_value_inequality:4} is satisfied for all \(k \leq m\). We shall choose \(\delta_m\) such that
	\begin{equation}\label{eq:thm:mean_value_inequality:5}
		\frac{C}{\delta_m \Lambda(C\delta_m^{-1} 4^{-m+1} \theta^{-2} H)} = \frac{1}{16}
	\end{equation}
	is satisfied. Then it would follow from \zcref{eq:thm:mean_value_inequality:3,eq:thm:mean_value_inequality:4}
	\begin{equation*}
		H_{m+1} \leq \frac{H_m}{16} \leq \frac{H}{16^{m+1}}.
	\end{equation*}
	We rewrite \zcref{eq:thm:mean_value_inequality:5} by
	\begin{equation*}
		C \delta_m^{-1} 4^{-m+1} \theta^{-2} H = \omega(4^{-m-1}\theta^{-2}H),
	\end{equation*}
	which means \(\delta_m = 16C 4^{-m-1}\theta^{-2}H / \omega(4^{-m-1}\theta^{-2}H)\) will satisfy \zcref{eq:thm:mean_value_inequality:5}. \(\omega\) is the function from \zcref{eq:sec:isoperimetric_inequality:2}.\par
	\emph{Step 2:} We need to choose \(\theta > 0\) such that \zcref{eq:thm:mean_value_inequality:2} is satisfied, or, equivalently,
	\begin{equation*}
		\sum_{k=0}^\infty \delta_k \leq \frac{T}{2} \quad \text{and} \quad \sum_{k=0}^\infty \sqrt{\delta_k} \leq \frac{R}{2}.
	\end{equation*}
	We have for \(\xi = 4^{-k}\theta^{-2}H\)
	\begin{equation*}
		\sum_{k=0}^\infty \delta_k = \sum_{k=1}^\infty \frac{16 C (4^{-k}\theta^{-2}H)}{\omega(4^{-k}\theta^{-2}H)} \leq 16C \int_0^\infty \frac{4^{-k}\theta^{-2}H}{\omega(4^{-k}\theta^{-2}H)}dk \leq 16C \int_0^{\theta^{-2}H}\frac{1}{\omega(\xi)}d\xi
	\end{equation*}
	(note \(dk = -\frac{d\xi}{\xi\ln(4)}\) and \(\ln(4) > 1\)), and analogously
	\begin{equation*}
		\sum_{k=0}^\infty \sqrt{\delta_k} \leq \sum_{k=1}^\infty \frac{4 \sqrt{C4^{-k}\theta^{-2}H}}{\sqrt{\omega(4^{-k}\theta^{-2}H)}} \leq 4\sqrt{C}\int_0^{\theta^{-2}H} \frac{1}{\sqrt{\xi \omega(\xi)}}d\xi.
	\end{equation*}
	Thus, \zcref{eq:thm:mean_value_inequality:2} is satisfied if
	\begin{equation*}
		\int_0^{\theta^{-2}H} \frac{1}{\omega(\xi)}d\xi \leq cT \quad \text{and} \quad \int_0^{\theta^{-2}H} \frac{1}{\sqrt{\xi \omega(\xi)}}d\xi \leq cR
	\end{equation*}
	for \(c \leq 1/(32 C)\). Equivalently,
	\begin{equation*}
		\theta^{-2}H \leq V(T) \quad \text{and} \quad \theta^{-2}H \leq W(R).
	\end{equation*}
	Hence, we choose
	\begin{equation*}
		\theta^2 = \frac{H}{\min\{V(T), W(R)\}}
	\end{equation*}
	and \zcref{eq:thm:mean_value_inequality:1} follows. The statement about the heat equation follows by an analysis of the constants.
\end{proof}

\begin{corollary}\label{cor:mean_value_inequality_1}
	Let \(u \in C^\infty(\overline{\amalg})\) be a subparabolic function with respect to \(L\). If there is an isoperimetric inequality 
	\begin{equation*}
		\Lambda(v) = av^{-2/\nu}
	\end{equation*}		
	in the ball \(B(z,R)\), where \(\nu \) and \(a\) are two positive constants, then the following estimate holds:
	\begin{equation}\label{eq:cor:mean_value_inequality_1:1}
		u(z,T)_+^2 \leq \const_{\nu, p, U} \frac{a^{-\nu/2}}{\min \{\sqrt{T}, R\}^{\nu+2}} \int_\amalg u_+^2 d\overline{\mu}.	
	\end{equation}	
\end{corollary}
\begin{proof}
	At first, we have to compute the functions \(V(t)\) and \(W(r)\). We have
	\begin{equation*}
		\xi = \frac{v}{\Lambda(v)} = \frac{v^{(\nu+2)/\nu}}{a} \Leftrightarrow \omega(\xi) = v = a^{\nu/(\nu+2)} \xi^{\nu/(\nu+2)}.
	\end{equation*}
	Now we need to calculate the integrals from \zcref{eq:sec:isoperimetric_inequality:3}. We obtain
	\begin{equation*}
			\const_{p, U} t = \int_0^{V(t)} \frac{d\xi}{\omega(\xi)} = a^{-\nu/(\nu+2)} \int_0^{V(t)} \xi^{-\nu/(\nu+2)} d\xi = \frac{\nu+2}{2}a^{-\nu/(\nu+2)} V(t)^{2/(\nu+2)},
	\end{equation*}
	and thus
	\begin{equation*}
		V(t) = \const_{\nu, p, U} a^{\nu/2} t^{(\nu+2)/2}.
	\end{equation*}
	Similarly, we have
	\begin{equation*}
		\begin{split}
			\const_{p, U}r &= \int_0^{W(r)} \frac{d\xi}{\sqrt{\xi\omega(\xi)}} = a^{-\nu/(2\nu+4)} \int_0^{W(r)} \xi^{-(2\nu+2)/(2\nu+4)} d\xi\\
			&=  (\nu+2) a^{-\nu/(2\nu+4)} W(r)^{1/(\nu+2)},
		\end{split}
	\end{equation*}
	and therefore
	\begin{equation*}
		W(r) = \const_{\nu, p, U} a^{\nu/2}r^{\nu+2}.
	\end{equation*}
	\zcref{eq:cor:mean_value_inequality_1:1} follows immediately by plugging the values for \(V(T)\) and \(W(R)\) into \zcref{eq:thm:mean_value_inequality:1}.
\end{proof}

\begin{remark}
	We are for \(\nu := n\) and \(a=a(n)\) in the situation of 
	\zcref{ex:isoperimetric_inequality_cartan_hardamard}.
\end{remark}

\begin{corollary}\label{cor:mean_value_inequality_2}
	Let \(u \in C^\infty(\overline{\amalg})\) be a subparabolic function with respect to \(L\). If there is an isoperimetric inequality 
	\begin{equation*}
		\Lambda(v) = \frac{b}{R^2}\left(\frac{V(z,R)}{v}\right)^\beta
	\end{equation*}		
	in the ball \(B(z,R)\), where \(b\) and \(\beta\) are two positive constants, then the following estimate holds:
	\begin{equation}\label{eq:cor:mean_value_inequality_2:1}
		u(z,T)_+^2 \leq \const_{\beta,p, U}\frac{b^{-1/\beta} }{\min\{(T/R^2)^{1/\beta}, R^2/T\}} \frac{1}{\overline{\mu}(\amalg)} \int_{\amalg} u_+^2 d\overline{\mu}.
	\end{equation}	
\end{corollary}
\begin{proof}
	Let us apply \zcref{cor:mean_value_inequality_1} for \(a = b V(z,R)^\beta / R^2\) and \(\nu = 2/\beta\). Using \(V(z,R) T = \overline{\mu}(\amalg)\), we obtain \zcref{eq:cor:mean_value_inequality_2:1}.
\end{proof}

\begin{remark}
	Here, we are in the situation of \zcref{ex:isoperimetric_inequality_volume_doubling_poincare}. By the way, the form of \zcref{eq:cor:mean_value_inequality_2:1} is the reason why we call \zcref{thm:mean_value_inequality} a Mean Value inequality.
\end{remark}

\section{Harnack Inequality}\label{sec:harnack_inequality}

In this section, we derive the parabolic Harnack inequality. We first establish a weakened form of the Harnack inequality and then use it to obtain the desired result. Throughout this section, we assume that the standing assumptions from \zcref{subsec:known_results} are in force. 

\subsection{Weak Harnack Inequality}\label{subsec:weak_harnack_inequality}

In what follows we consider the cylinder
\begin{equation*}
	\amalg_{2R} := B(z,2R) \times (0, 4R^2)
\end{equation*}
and the subcylinders \(\amalg_{2R, +}, \amalg_{2R, -} \subset \amalg\) defined by
\begin{equation*}
	\amalg_{2R,+} := B(z,R) \times (3R^2, 4R^2) \quad \text{and} \quad \amalg_ {2R,-} := B(z,R) \times (R^2, 2R^2).
\end{equation*}

\begin{theorem}[Weak Harnack Inequality]\label{thm:weak_harnack_inequality}
	Let \(u  \in C^\infty(\overline{\amalg}_{2R})\) be a non-negative superparabolic function with respect to \(L\). We fix \(\kappa > 0\) and set
	\begin{equation*}
		H := \left\{ (x,t) \in \amalg_R : u(x,t) \geq \kappa\right\}.
	\end{equation*}
	Then, for any \(\delta > 0\), there exists \(\varepsilon = \varepsilon(\delta, A, a, N, p, U) > 0\) such that if 
	\begin{equation*}
		\overline{\mu}(H) \geq \delta \overline{\mu} (\amalg_R),
	\end{equation*}
	then
	\begin{equation}\label{eq:thm:weak_harnack_inequality:1}
		\inf_{\amalg_{2R, +}} u \geq \varepsilon \kappa.
	\end{equation}
\end{theorem}
\begin{proof} We may assume without loss of generality \(\kappa = 1\) by replacing \(u\) by \(u / \kappa\). Moreover, it suffices to prove the theorem for positive superparabolic functions \(u\). Indeed, let us assume \(\inf_{\amalg_{2R}} u = 0\) and let \(m > 0\) be an arbitrary positive real number. We note
	\begin{equation*}
		u \geq 1 \Leftrightarrow u+m \geq 1+m,
	\end{equation*}		
	so we obtain according to \zcref{eq:thm:weak_harnack_inequality:1} for the positive superparabolic function \(u+m\)
	\begin{equation*}
		\inf_{\amalg_{2R,+}} (u+m) \geq \varepsilon(1+m).
	\end{equation*}
	\zcref[S]{eq:thm:weak_harnack_inequality:1} follows by passing \(m \to 0\).\par
	Let us consider the function
	\begin{equation*}
		v = \ln\left(\frac{1}{u}\right) = -\ln u.
	\end{equation*}
	In terms of this function, we have
	\begin{equation*}
		H := \left\{ (x,t) \in \amalg_R : v(x,t) \leq 0 \right\} \quad \text{and} \quad \overline{\mu}(H) \geq \delta \overline{\mu} (\amalg_R),
	\end{equation*}
	and we need to prove
	\begin{equation}\label{eq:thm:weak_harnack_inequality:2}
		\sup_{\amalg_{2R, +}} v \leq \const_{\delta, A, a, N, p, U}.
	\end{equation}\par
	\emph{Step 1:} We show that the function \(v\) is subparabolic in \(\amalg_{2R}\) with respect to \(L\). Let us check the Neumann condition. Let \(n\) be the outward-pointing unit normal vector field along \(\partial M\). Then
	\begin{equation*}
		\langle \nabla v, n\rangle_U \restrict{\partial M \cap B(z,2R)} = -\frac{1}{u} \langle \nabla u, n \rangle_U \restrict{\partial M \cap B(z,2R)} = 0.
	\end{equation*}
	Furthermore, rewriting the expression \(Lv\) leads to
	\begin{equation*}
		Lv = - \frac{Lu}{u} - \frac{\|\nabla u\|_{U}^2}{u^2} = - \frac{Lu}{u} - \|\nabla v\|_U^2 \leq 0.
	\end{equation*}		
	\par
	\emph{Step 2:} Let \(\eta \in C_0^\infty(B(z, 2R))\). We claim that for all \(t \in (0, 4R^2)\)
	\begin{equation}\label{eq:thm:weak_harnack_inequality:3}
		\int_{B(z, 2R)} p\partial_t (v_+) \eta^2 d\mu \leq -\frac{1}{2} \int_{B(z,2R)} \|\nabla v_+\|_U^2 \eta^2 d\mu + 2 \int_{B(z,2R)} \|\nabla \eta\|_U^2 d\mu.
	\end{equation}
	Fix \(t \in (0, 4R^2)\), let \(\theta > 0\) be a regular value of both \(v(-,t)\) and \(v(-, t)\restrict{\partial M}\) and define
	\begin{equation*}
		\Omega_\theta := \{ x \in B(z, 2R) : v(x,t) > \theta \}.
	\end{equation*}
	\begin{figure}
		\centering
		\begin{tikzpicture}
  			\draw[thick, dashed, green!80!black] (0,1) circle (2);
  			
  			\begin{scope}
    			\clip (-4,-3) .. controls (0,1) .. (-4,4) -- (4,4) -- (4,-3) -- cycle;
    			\clip (0,1) circle (2);
    			\clip (0.5,-4) circle (5);
    			\fill[gray!10] (-7,-7) rectangle (7,7);
  			\end{scope}

  			\begin{scope}
    			\clip (-4,-3) .. controls (0,1) .. (-4,4) -- (4,4) -- (4,-3) -- cycle;
    			\fill[blue!20, opacity=0.25] (0.5,-4) circle (5);
    			\fill[green!20, opacity=0.25] (0,1) circle (2);
    			\draw[thick, blue!80!black] (0.5,-4) circle (5);
    			\draw[thick, green!80!black] (0,1) circle (2);
  			\end{scope}
  
  			\node[gray!25!black] at (0.4,-0.4) { \(\Omega_\theta\)};
  			\node[blue!50!black] at (0.7,-2.5) { \(\{ v > \theta\}\)};
  			\node[green!50!black] at (0.5,1.8) { \(B(z, 2R)\)};
  			\node at (2.1,2.8) { \(M\)};
  			\draw[black, thick] (-4,-3) .. controls (0,1) .. (-4,4) node[pos=0.9, below left, inner sep=1pt] { \(\partial M\)};
\end{tikzpicture}
	\caption{The set \(\Omega_\theta\) and the three contributing boundary pieces}
	\end{figure}
	We obtain by \zcref{cor:integration_by_parts}
	\begin{equation*}
        \begin{split}
            \int_{\Omega_\theta} p \partial_t v \eta^2 d\mu & \leq \int_{\Omega_\theta} \divg(U \nabla v) \eta^2 d\mu - \int_{\Omega_\theta} \|\nabla v\|_U^2 \eta^2 d\mu\\
            &= -2 \int_{\Omega_\theta} \langle \nabla v, \nabla \eta\rangle_U \eta d\mu - \int_{\Omega_\theta} \|\nabla v\|_U^2 \eta^2 d\mu + \int_{\partial \Omega_\theta \cap M} \langle \nabla v, n\rangle_U \eta^2 d\widetilde{\mu}.
        \end{split}
	\end{equation*}
	Our goal is to show that the integral over the boundary is non-positive. We have
	\begin{equation*}
		\partial \Omega_\theta \cap M \subset \partial B(z, 2R) \cup \{v=\theta\} \cup \partial M.
	\end{equation*}
	Note that \(\eta^2 = 0\) in a neighbourhood of \(\partial B(z, 2R)\). Moreover, on \(\partial \Omega_\theta \cap \partial M\) the outward-pointing unit normal of \(\partial M\) coincides with that of \(\partial \Omega_\theta \cap M\), hence the Neumann condition yields \(\langle \nabla v, n\rangle_U = 0\). It remains to examine the behaviour on \(\{v = \theta\}\). Since the gradient \(\nabla v\) is normal to the level set, and because \(v > \theta\) inside \(\Omega_\theta\), the vector field \(\nabla v\) points inward. Thus the outward-pointing unit normal along \(\{v = \theta\}\) is given by
	\begin{equation*}
		n = -\frac{\nabla v}{|\nabla v|}.
	\end{equation*}
	Consequently,
	\begin{equation*}
		\langle \nabla v,n\rangle_U = -\frac{1}{|\nabla v|} \|\nabla v\|_U^2 \leq 0,
	\end{equation*}
	and therefore the integral over the boundary is non-positive. We arrive at
	\begin{equation*}
		\int_{\Omega_\theta} p\partial_t v \eta^2 d\mu \leq -2 \int_{\Omega_\theta} \langle\nabla v, \nabla\eta\rangle_U \eta d\mu - \int_{\Omega_\theta} \|\nabla v\|_U^2 \eta^2 d\mu.
	\end{equation*}
	Note
	\begin{equation*}
		-2\langle \nabla v, \nabla \eta\rangle_U \eta\leq 2\|\nabla v\|_U \eta\|\nabla \eta\|_U \leq \frac{1}{2} \|\nabla v\|_U^2\eta^2 + 2\|\nabla \eta\|_U^2,
	\end{equation*}
	where the first inequality follows from the Cauchy-Schwarz inequality and the second from Young's inequality\footnote{More precisely, we used the following consequence of Young's inequality: \(2|ab| \leq ta^2+\frac{1}{t}b^2\) for all \(a,b \in \R\) and \(t > 0\). Alternatively, it follows by rearranging \(0 \leq \|\eta \nabla v + 2\nabla \eta\|_U
	^2\).}.
	Combining the above estimates, we obtain
	\begin{equation*}
			\int_{\Omega_\theta} p\partial_t v \eta^2 d\mu \leq -\frac{1}{2}\int_{\Omega_\theta} \|\nabla v\|_U^2\eta^2d\mu + 2\int_{\Omega_\theta} \|\nabla \eta\|_U^2 d\mu
	\end{equation*}
	Letting \(\theta \to 0\), which is justified by the Dominated Convergence Theorem, yields \zcref{eq:thm:weak_harnack_inequality:3}.\par
	\emph{Step 3:} Now we generalize \zcref{eq:thm:weak_harnack_inequality:3} to Lipschitz functions supported in \(\overline{B(z,2R)}\). We fix a real number \(r < R\) and a smooth cutoff function \(\varphi\) of \(B(z, 2r)\) in \(B(z, 2R)\) such that \(\varphi \eta \in W_c^{1,2}(B(z, 2R)) \subset W_0^{1,2}(B(z,2R))\) and rename \(\varphi \eta\) back into \(\eta\). Consequently, we find a sequence of smooth functions \(\eta_k \in C_0^\infty(B(z, 2R))\) with
	\begin{equation*}
		\eta_k \xrightarrow{W^{1,2}} \eta \quad \text{as} \ k \to \infty.
	\end{equation*}		
	 We obtain
	\begin{equation*}
		\begin{split}
			\left|\int_ {B(z,2r)} p \partial_t (v_+) \left(\eta^2 - \eta_k^2\right)d\mu\right| &\leq \const \|\eta^2 - \eta_k^2\|_{L^1(B(z,2r))}\\
			&\leq \const \|\eta + \eta_k\|_{L^2(B(z,2r))}\|\eta - \eta_k\|_{L^2(B(z,2r))}\\
			&\xrightarrow{k \to \infty} 0
		\end{split}
	\end{equation*}
	and similarly
	\begin{equation*}
		\left| \int_{B(z,2r)} \|\nabla v_+\|_U^2 (\eta^2-\eta_k^2)d\mu\right| \xrightarrow{k\to\infty} 0.
	\end{equation*}
	For the second integral on the right-hand side of \zcref{eq:thm:weak_harnack_inequality:3} we note that
	\begin{equation*}
		\int_{B(z,2r)} \|\nabla \eta - \nabla \eta_k\|_U^2	d\mu \simeq \int_{B(z,2r)} |\nabla \eta - \nabla \eta_k|^2	d\mu \xrightarrow{k \to \infty} 0, 
	\end{equation*}
	hence
	\begin{equation*}
		\int_{B(z,2r)} \|\nabla \eta_k\|_U^2d\mu \xrightarrow{k \to \infty} \int_{B(z,2r)} \|\nabla \eta\|_U^2d\mu
	\end{equation*}
	follows. The desired claim follows by letting \(r \to R\) using the Dominated Convergence Theorem.\par
	Let us set \(\eta(x) = (d(x)/2R)^{\alpha/2}\), where \(d(x) = \dist(x, \partial B(z,2R))\) and \(\alpha\) is the constant from \zcref{thm:poincare-type_inequality}. The inequality
	\begin{equation*}
		|\nabla \eta| = \frac{\alpha}{4R} \left(\frac{d}{2R}\right)^{\alpha/2 - 1} |\nabla d| \leq \frac{\alpha}{4R}
	\end{equation*}
	follows for \(\alpha > 2\) (and we may choose \(\alpha\) large enough), therefore we arrive at
	\begin{equation}\label{eq:thm:weak_harnack_inequality:4}
		\int_{B(z,2R)} |\nabla \eta|^2 d\mu \leq \frac{\alpha^2}{16R^2} V(z,2R).
	\end{equation}\par
	\emph{Step 4:} The function
	\begin{equation*}
		I(t) := \int_{B(z,2R)} pv_+(\cdot, t) \eta^2 d\mu
	\end{equation*}
	satisfies for all \(t \in [R^2, 4R^2]\) the inequality
	\begin{equation}\label{eq:thm:weak_harnack_inequality:5}
		I(t) \leq \const_{A,a,N, p} \frac{V(z, 2R)}{\delta}.
	\end{equation}
	Let 
	\begin{equation*}
		H_t := \{x \in B(z,R) : v(x,t) \leq 0\}.
	\end{equation*}
	According to \zcref{thm:poincare-type_inequality} we have
	\begin{equation}\label{eq:thm:weak_harnack_inequality:6}
		\int_{B(z,2R)} |\nabla v_+|^2\eta^2 d\mu \geq \const_{A,a,N} \frac{\mu(H_t)}{R^2V(z,2R)^2} \left(\int_{B(z,2R)} v_+\eta^2\right)^2.
	\end{equation}
	We deduce from \zcref{eq:thm:weak_harnack_inequality:3,eq:thm:weak_harnack_inequality:4,eq:thm:weak_harnack_inequality:6} the relation
	\begin{equation}\label{eq:thm:weak_harnack_inequality:7}
        \begin{split}
            \frac{d}{dt}I(t) &= \int_{B(z,2R)} p\partial_t(v_+(\cdot, t)) \eta^2 d\mu \leq -\frac{1}{2}\int_{B(z,2R)} \|\nabla v_+\|_U^2 \eta^2 d\mu + 2\int_{B(z,2R)}\|\nabla \eta\|_U^2 d\mu\\
            &\leq -K(t)I(t)^2+D,
        \end{split}
	\end{equation}
	where
	\begin{equation*}
		K(t) = \const_{A,a,N,p, U} \frac{\mu(H_t)}{V(z,2R)^2R^2} \quad\text{and}\quad D = \const_{p, U}\frac{\alpha^2V(z,2R)}{R^2}.
	\end{equation*}
	Now we are going to prove the auxiliary estimate
	\begin{equation}\label{eq:thm:weak_harnack_inequality:8}
		I(t) \leq \left( \int_0^{R^2} K(\tau) d\tau \right)^{-1} + Dt
	\end{equation}
	for all \(R^2 \leq t \leq 4R^2\). If \(I(t^*) \leq Dt^*\) for some \(t^* \leq R^2\), then we obtain for \(t \geq t^*\) the inequality \(I(t) \leq Dt\) because \((d/dt)I \leq D\). \zcref{eq:thm:weak_harnack_inequality:8} follows in this case. Suppose \(I(t) > Dt\) for all \(t \leq R^2\). We define \(J(t) = I(t) - Dt\), and \zcref{eq:thm:weak_harnack_inequality:7} yields \((d/dt)J \leq -KJ^2\). We obtain
	\begin{equation*}
		\frac{d}{dt}\left(-\frac{1}{J(t)}\right) = \frac{1}{J^2(t)} \frac{d}{dt} J(t) \leq -K(t),
	\end{equation*}
	and integrating from \(0\) to \(R^2\) yields
	\begin{equation*}
		-\frac{1}{J(R^2)} \leq \frac{1}{J(0)} - \frac{1}{J(R^2)} = \int_0^{R^2} \frac{d}{d\tau}\left(-\frac{1}{J(\tau)}\right)d\tau \leq -\int_0^{R^2} K(\tau)d\tau.
	\end{equation*}
	Hence, we arrive at the inequality
	\begin{equation*}
		J(R^2) \leq \left(\int_0^{R^2} K(\tau) d\tau\right)^{-1}.
	\end{equation*}
	Since for \(t > R^2\) we have by the fundamental theorem of calculus
	\begin{equation*}
		I(t) \leq I(R^2) + D(t-R^2) = J(R^2)+Dt \leq \left( \int_0^{R^2} K(\tau) d\tau \right)^{-1} + Dt,
	\end{equation*}
	from this and the preceding estimate for \(J(R^2)\) we obtain \zcref{eq:thm:weak_harnack_inequality:8}. We note
	\begin{equation*}
		\int_0^{R^2} \mu(H_\tau) d\tau = \overline{\mu}(H) \geq \delta \overline{\mu}(\amalg_R) = \delta R^2 V(z,R) \geq \frac{\delta R^2}{A}V(z,2R),
	\end{equation*}
	and substituting the values of \(K(\tau)\) and \(D\) into \zcref{eq:thm:weak_harnack_inequality:8}
	\begin{equation*}
		\begin{split}
			I(t) &\leq \const_{A,a,N,p, U} R^2 V(z,2R)^2 \left(\int_0^{R^2} \mu(H_\tau)d\tau\right)^{-1} + \const_{p, U} \alpha^2 V(z,2R)\\
			&\leq \const_{A,a,N,p, U} \frac{A}{\delta} V(z,2R) + \const_{p, U} \frac{\alpha^2 V(z,2R)}{\delta}\\
			&\leq \const_{A,a,N,p, U} \frac{V(z,2R)}{\delta},
		\end{split}
	\end{equation*}
	where we have used \(t \leq 4R^2\) and \(\delta \leq 1\) as well. This is exactly \zcref{eq:thm:weak_harnack_inequality:5}.\par
	\emph{Step 5:} Now we finally prove \zcref{eq:thm:weak_harnack_inequality:1}.
	Integrating \zcref{eq:thm:weak_harnack_inequality:3} with respect to \(t\) from \(R^2\) to \(4R^2\), applying \zcref{eq:thm:weak_harnack_inequality:5} and using the uniform ellipticity of \(U\) leads to
	\begin{equation*}
		\begin{split}
			\int_{B(z,2R)} p\eta^2 [v_+(\cdot, t)]_{t=R^2}^{t=4R^2} d\mu
			&\leq -\frac{1}{2} \int_{R^2}^{4R^2} \int_{B(z,2R)} \|\nabla v_+\|_U^2 \eta^2 d\mu dt + 6R^2 \int_{B(z,2R)} \|\nabla \eta\|_U^2 d\mu\\
			&\leq -\frac{1}{2} \int_{R^2}^{4R^2} \int_{B(z,2R)} \|\nabla v_+\|_U^2 \eta^2 d\mu dt + \const_{A,a,N,U} V(z,2R).
		\end{split}
	\end{equation*}
	Since \(\eta^2\restrict{B(z, 5R/3)} \geq 1/6^\alpha\) and \(\delta \leq 1\), we obtain from rearranging the previous equation, uniform ellipticity of \(U\) and \zcref{eq:thm:weak_harnack_inequality:5} the estimate
	\begin{equation}\label{eq:thm:weak_harnack_inequality:9}
		\begin{split}
			\int_{R^2}^{4R^2} &\int_{B(z, 5R/3)} \|\nabla v_+\|_U^2 d\mu dt \leq 6^\alpha \int_{R^2}^{4R^2}\int_{B(z, 5R/3)} \|\nabla v_+\|_U^2\eta^2 d\mu dt\\
			&\leq 2 \cdot 6^\alpha \int_{B(z,2R)} p \eta^2 v_+(\cdot, R^2)d\mu + \const_{A,a,N,p,U}V(z, 2R)\\
			&\leq \const_{A,a,N,p, U}\frac{V(z, 2R)}{\delta} .
		\end{split}
	\end{equation}
	We set
	\begin{equation*}
		\overline{v}(t) := \fint_{B(z, 4R/3)} v_+(x,t)d\mu(x),
	\end{equation*}
	and \zcref{eq:thm:weak_harnack_inequality:5} yields
	\begin{equation}\label{eq:thm:weak_harnack_inequality:10}
		\overline{v}(t) \leq \frac{\const_p 6^\alpha}{V(z,4R/3)} I(t) \leq \frac{\const_{A,a,N, p}}{\delta}
	\end{equation}
	for all \(t \in [R^2, 4R^2]\). Now we apply \zcref{eq:thm:poincare_1:2} from \zcref{thm:poincare_1} to get
	\begin{equation*}
		\begin{split}
			\int_{B(z, 4R/3)} v_+^2d\mu &= \int_{B(z, 4R/3)} (\overline{v} -(v_+ - \overline{v}))^2d\mu\leq 2\int_{B(z,4R/3)} \overline{v}^2d\mu + 2\int_{B(z,4R/3)} (v_+ - \overline{v})^2d\mu\\
			&\leq 2 \overline{v}^2 V(z, 4R/3) + \const_{A,a,N} R^2 \int_{B(z, 5R/3)} |\nabla v_+|^2d\mu\\
			&\leq 2 \overline{v}^2 V(z, 4R/3) + \const_{A,a,N,U} R^2 \int_{B(z, 5R/3)} \|\nabla v_+\|_U^2d\mu.
		\end{split}
	\end{equation*}
	We integrate this inequality with respect to \(t\) from \(R^2\) to \(4R^2\) and use \zcref{eq:thm:weak_harnack_inequality:9,eq:thm:weak_harnack_inequality:10} to derive
	\begin{equation}\label{eq:thm:weak_harnack_inequality:11}
			\int_{R^2}^{4R^2} \int_{B(z, 4R/3)} v_+^2 d\mu dt\leq \const_{A, a, N,p, U}\frac{R^2}{\delta^2} V(z, 2R).
	\end{equation}
	Here we have also used \(\delta \leq 1\) and \(V(z, 4R/3) \leq V(z, 2R)\). Due to \zcref{ex:isoperimetric_inequality_volume_doubling_poincare} we may apply \zcref{cor:mean_value_inequality_2} to the function \(v\) in the cylinders
	 \begin{equation*}
	 	B(x, R/3) \times (t-(R/3)^2, t) \subset B\left(z, \frac{4R}{3}\right) \times (R^2, 4R^2),
	 \end{equation*}
	 for  \((x,t) \in \amalg_{2R,+}\). We obtain
	 \begin{equation*}
	    \begin{split}
	 		v_+(x,t)^2 &\leq \frac{\const_{A, a, N, p, U}}{R^2 V(x, R/3)} \int_{t-R^2/9}^t \int_{B(x, R/3)} v_+^2 d\mu dt \leq \frac{\const_{A, a, N, p, U}}{R^2 V(z, 2R)} \int_{R^2}^{4R^2} \int_{B(z, 4R/3)} v_+^2 d\mu dt
 		\end{split}
	 \end{equation*}
	 and together with \zcref{eq:thm:weak_harnack_inequality:11} this implies \(v(x,t) \leq \const_{A, a, N, p, U}/\delta\). So \zcref{eq:thm:weak_harnack_inequality:2} is proven.
\end{proof}

\begin{corollary}[Alternative Form of the Weak Harnack Inequality]\label{cor:alt_weak_harnack_inequality}
Let \(u \in C^\infty(\overline{\amalg}_{2R})\) be a subparabolic function in the cylinder \(\amalg_{2R}\) and set
	\begin{equation*}
		H := \{ (x,t) \in \amalg_R : u(x,t) \leq 0\}.
	\end{equation*}
	For any \(\delta > 0\) there exists \(\varepsilon = \varepsilon(\delta, A, a, N, p, U) > 0\) such that if
	\begin{equation*}
		\overline{\mu}(H) \geq \delta \overline{\mu}(\amalg_R),
	\end{equation*}
	then
	\begin{equation}\label{eq:cor:alt_weak_harnack_inequality:1}
		\sup_{\amalg_{2R}} u \geq (1+\varepsilon) u(z, 4R^2).
	\end{equation}
\end{corollary}
\begin{proof}
	If \(u(z, 4R^2) \leq 0\), the \zcref{eq:cor:alt_weak_harnack_inequality:1} is obviously satisfied. So we may assume \(u(z, 4R^2) > 0\). By rescaling we assume without loss of generality
	\begin{equation*}
		\sup_{\amalg_{2R}} u = 1.
	\end{equation*}
	Consider the non-negative superparabolic function \(v = 1-u\) in \(\amalg_{2R}\). By construction we have
	\begin{equation*}
		H = \{ (x,t) \in \amalg_R : v(x,t) \geq 1\}.
	\end{equation*}
	The condition \(\overline{\mu}(H) \geq \delta\overline{\mu}(\amalg_R)\) implies according to \zcref{thm:weak_harnack_inequality}
	\begin{equation*}
		\inf_{\amalg_{2R,+}} v \geq \varepsilon = \varepsilon(\delta, A, a, N, p, U).
	\end{equation*}
	It follows that \(v(z, 4R^2) \geq \varepsilon\), and hence
	\begin{equation*}
		u(z, 4R^2) \leq 1-\varepsilon \leq \frac{1}{1+\varepsilon} = \frac{1}{1+\varepsilon} \sup_{\amalg_{2R}} u
	\end{equation*}
	follows. This estimate is equivalent to \zcref{eq:cor:alt_weak_harnack_inequality:1}.
\end{proof}

\subsection{Harnack Inequality}\label{subsec:harnack_inequality}

In this section we establish the Harnack inequality. At first glance it may seem unusual to begin with a result called the Weak Harnack inequality. The reason is that the full Harnack inequality can be derived from a careful application of this weaker version, while the Weak Harnack Inequality itself rests on the results stated in \zcref{subsec:known_results}. Some of our figures in this chapter are inspired by those in \cite{grigoryan1991heat}.

\begin{lemma}[Iterated Weak Harnack Inequality]\label{lem:iterated_weak_harnack_inequality}
	Let \(u \in C^\infty(\overline{\amalg}_{2R})\) be a non-negative superparabolic function in \(\amalg_{2R}\). Let
	\begin{equation*}
		\amalg^0 = B(y,r) \times (\tau, \tau + r^2) \subset \amalg_{R/2}
	\end{equation*}
	and
	\begin{equation*}
		H := \{ (x,t) \in \amalg^0 : u(x,t) > 1\}.
	\end{equation*}
	If
	\begin{equation*}
		\overline{\mu}(H) \geq \delta \overline{\mu}(\amalg^0)
	\end{equation*}
	for a real number \(\delta > 0\), then
	\begin{equation}\label{eq:lem:iterated_weak_harnack_inequality:1}
		u(z,4R^2) \geq \const_{\delta, A, a, N, p, U} \left( \frac{\overline{\mu}(\amalg^0)}{\overline{\mu}(\amalg_R)}\right)^l,
	\end{equation}
	where \(l = l(A, a, N, p, U) > 0\) is an absolute constant.
\end{lemma}
\begin{proof}
	We consider the cylinders
	\begin{equation*}
		\amalg^k = B(y, 2^kr) \times (\tau, \tau + 4^kr^2) \quad \text{and} \quad \widetilde{\amalg}^k = B(y, 2^kr) \times (\tau + 3\cdot 4^k r^2, \tau + 4^{k+1}r^2),
	\end{equation*}
	where \(k \geq 0\) is an integer.
	\begin{figure}
		\centering
		\begin{tikzpicture}
  			\fill[blue!30] (-0.75,-2.5) arc(180:360:0.75 and 0.225)
   			 	-- (0.75,-1.5) arc(0:180:0.75 and 0.225)
   			 	-- cycle;
  			\fill[blue!20] (0,-1.5) ellipse (0.75 and 0.225);
  			\fill[blue!20] (0,-2.5) ellipse (0.75 and 0.225);
  			\draw (-0.75,-2.5) -- (-0.75,-1.5);
  			\draw ( 0.75,-2.5) -- ( 0.75,-1.5);
  			\draw (0,-1.5) ellipse (0.75 and 0.225);
  			\draw[dashed] ( 0.75,-2.5) arc(0:180:0.75 and 0.225);
  			\draw (-0.75,-2.5) arc(180:360:0.75 and 0.225);
  		
  			\fill[blue!30] (-0.75,-4.5) arc(180:360:0.75 and 0.225)
   			 	-- (0.75,-3.5) arc(0:180:0.75 and 0.225)
   			 	-- cycle;
  			\fill[blue!20] (0,-3.5) ellipse (0.75 and 0.225);
  			\fill[blue!20] (0,-4.5) ellipse (0.75 and 0.225);
  			\draw (-0.75,-4.5) -- (-0.75,-3.5);
  			\draw ( 0.75,-4.5) -- ( 0.75,-3.5);
  			\draw (0,-3.5) ellipse (0.75 and 0.225);
  			\draw[dashed] ( 0.75,-4.5) arc(0:180:0.75 and 0.225);
  			\draw (-0.75,-4.5) arc(180:360:0.75 and 0.225);

 			\draw (-1.5,-4.5) -- (-1.5,-1.5);
 			\draw ( 1.5,-4.5) -- ( 1.5,-1.5);
  			\draw (0,-1.5) ellipse (1.5 and 0.375);
  			\draw (-1.5,-4.5) arc(180:360:1.5 and 0.375);
  			\draw[dashed] (1.5,-4.5) arc(0:180:1.5 and 0.375);
  
  			\draw (-3,-6) -- (-3,0);
			\draw (3,-6) -- ( 3,0);
			\draw (0,0) ellipse (3 and 0.75);
			\draw (-3,-6) arc(180:360:3 and 0.75);
			\draw[dashed] (3,-6) arc(0:180:3 and 0.75);
	
			\fill (0,-4.5) ellipse (1.5pt and 0.375pt);
			\node at  (2,-6) { \(\amalg_{2R}\)};
			\draw [->,blue!70!black] (2.2, -1.2) -- (0.5,-2) node[pos=0, right, inner sep=1pt] { \(\widetilde{\amalg}^k\)};
			\draw [->,blue!70!black] (2.2, -4.5) -- (0.5,-4) node[pos=0, right, inner sep=1pt] { \(\amalg^k\)};
			\draw[->] (-1.2,-4.8) -- (0,-4.5) node[pos=0, below, inner sep=1pt] { \((y,\tau)\)};
			\draw [->] (-1.75 ,-3) -- (-1,-3) node[pos = 0, left, inner sep =1pt] { \(\amalg^{k+1}\)};
		\end{tikzpicture}
		\caption{Subcylinders \(\amalg^k, \widetilde{\amalg}^k \subset \amalg^{k+1}\)}
	\end{figure}		
	 Note that \(\amalg^k \subset \amalg^{k+1}\) and \(\widetilde{\amalg}^k \subset \amalg^{k+1}\). We apply \zcref{thm:weak_harnack_inequality} successively in these cylinders. For \(\amalg^0, \widetilde{\amalg}^0 \subset \amalg^1\) we obtain the lower bound
	\begin{equation*}
		\inf_{\widetilde{\amalg}^0} u \geq \varepsilon_1 = \varepsilon_1(\delta, A, a, N, p, U).
	\end{equation*}
	Now we consider the function \(u_1 = u/\varepsilon_1\). Note \(u_1\restrict{\widetilde{\amalg}^0} \geq 1\) and for any \(k \geq 1\)
	\begin{equation*}
		\overline{\mu}(\amalg^{k-1}) = V(y, 2^{k-1}r) \cdot 4^{k-1} r^2 \geq \frac{V(y, 2^kr) \cdot 4^kr^2}{4A} = \frac{\overline{\mu}(\amalg^k)}{4A}.
	\end{equation*}
	We obtain by \zcref{thm:weak_harnack_inequality}
	\begin{equation*}
		\inf_{\widetilde{\amalg}^1} u_1 \geq \varepsilon_2 = \varepsilon_2(A, a, N, p, U).
	\end{equation*}
	By replacing \(\varepsilon_2\) by \(\min\{\varepsilon_2, 1/2\}\), we may assume that \(0 < \varepsilon_2 < 1\). Let us now define for each \(k \geq 2\) functions \(u_k := u/(\varepsilon_1\varepsilon_2^{k-1}) = u_{k-1}/\varepsilon_2\) and assume inductively that
	\begin{equation*}
		\inf_{\widetilde{\amalg}^{k-1}} u_{k-1} \geq \varepsilon_2.
	\end{equation*}	 
	We obtain \(u_k \restrict{\widetilde{\amalg}^{k-1}} \geq 1\), and get by \zcref{thm:weak_harnack_inequality}
	\begin{equation*}
		\inf_{\widetilde{\amalg}^k} u_k \geq \varepsilon_2.
	\end{equation*}
	Rearranging terms leads to
	\begin{equation*}
		\inf_{\widetilde{\amalg}^k} u \geq \varepsilon_1 \varepsilon_2^k.
	\end{equation*}
	Now let \(K\) be the largest integer for which \(\amalg^K \subset \amalg_R\) holds. Since \(\amalg^0 \subset \amalg_{R/2}\), we have \(\amalg^1 \subset \amalg_R\), and hence \(K \geq 1\). The estimate we have obtained so far is satisfied, in particular, for \(k = K\). Now we show the auxiliary estimate 
	\begin{equation}\label{eq:lem:iterated_weak_harnack_inequality:2}
		2^{K+1}r \geq \frac{R}{2}.
	\end{equation}
	Assume for the contrary that \(2^{K+1}r < R/2\). This implies \(\tau + 4^{K+1}r^2 > R^2\) because \(\amalg^{K+1} \not\subset \amalg_R\). Since \(\tau \leq \tau + r^2 \leq (R/2)^2\), we must have
	\begin{equation*}
		2^{K+1}r \geq \sqrt{R^2 - \tau} \geq \sqrt{R^2 - \frac{R^2}{4}} = \frac{\sqrt{3} R}{2} \geq \frac{R}{2}.
	\end{equation*}
	A contradiction! So \zcref{eq:lem:iterated_weak_harnack_inequality:2} holds in fact. We obtain
	\begin{equation*}
		\overline{\mu}(\amalg^K) = 4^K r^2 V(y, 2^K r) \geq \frac{R^2}{16} V(y, R/4) \geq \frac{1}{16A^3} \overline{\mu}(\amalg_R).
	\end{equation*}
    Since \(\widetilde{\amalg}^{K-1} \subset \amalg^K \subset \amalg_R\), the estimate above yields
    \begin{equation*}
        \inf_{\widetilde{\amalg}^{K-1}} u \geq \varepsilon_1 \varepsilon_2^{K-1}.
    \end{equation*}
    Moreover, we have
    \begin{equation}
        \overline{\mu}(\widetilde{\amalg}^{K-1}) = \overline{\mu}(\amalg^{K-1}) \geq \frac{1}{4A} \overline{\mu}(\amalg^K) \geq \frac{1}{64A^4} \overline{\mu}(\amalg_R).
    \end{equation}
    Hence, applying \zcref{thm:weak_harnack_inequality}  to the function \(w := u/(\varepsilon_1 \varepsilon_2^{K-1})\) in the cylinders \(\amalg_R \subset \amalg_{2R}\), we obtain
    \begin{equation*}
        \inf_{\amalg_{2R,+}} u \geq \varepsilon_1 \varepsilon_2^{K-1} \varepsilon_3 \geq \varepsilon_1 \varepsilon_2^K \varepsilon_3, 
    \end{equation*}
    where \(\varepsilon_3 = \varepsilon_3(A,a,N,p,U) > 0\). Since \(\overline{\mu}(\amalg^0) 4^K \leq \overline{\mu}(\amalg^K) \leq \overline{\mu}(\amalg_R)\), it follows that
    \begin{equation*}
        K \leq \log_4\left(\frac{\overline{\mu}(\amalg_R)}{\overline{\mu}(\amalg^0)}\right).
    \end{equation*}
    Therefore,
    \begin{equation*}
		\inf_{\amalg_{2R, +}} u \geq \varepsilon_1\varepsilon_3 \left(\frac{\overline{\mu}(\amalg^0)}{\overline{\mu}(\amalg_R)}\right)^{-\log_4 \varepsilon_2},
	\end{equation*}
	which proves \zcref{eq:lem:iterated_weak_harnack_inequality:1} with \(c = \varepsilon_1\varepsilon_3\) and \(l = -\log_4 \varepsilon_2\).
\end{proof}

\begin{lemma}[Lemma of Growth]\label{lem:lemma_of_growth}
	Let \(u \in C^\infty(\overline{\amalg}_R)\) be a non-negative parabolic function in the cylinder \(\amalg_R\) and set
	\begin{equation*}
		E = \{(x,t) \in \amalg_R : u(x,t) > \kappa \}
	\end{equation*}	 
	for \(\kappa > 0\). Assume that
	\begin{equation*}
		 u(z, R^2) \geq 2\kappa,
	\end{equation*}
	then there exists \(\eta = \eta(A,a,N, p, U) > 0\) such that
	\begin{equation*}
		\overline{\mu}(E) \leq \eta \overline{\mu}(\amalg_R)
	\end{equation*}
	implies
	\begin{equation}\label{eq:lem:Lemma_of_Growth:1}
		\sup_{\amalg_R} u \geq 4\kappa.
	\end{equation}
\end{lemma}
\begin{proof}
	By replacing \(u\) by \(u/\kappa\) we may assume without loss of generality \(\kappa = 1\). We set \(\delta = 1/2\) in \zcref{thm:weak_harnack_inequality} and fix \(\varepsilon = \varepsilon(\delta, A, a, N, p, U)\). The idea is to apply \zcref{cor:alt_weak_harnack_inequality} to the function \(u-1\) to obtain the estimate
	\begin{equation}\label{eq:lem:Lemma_of_Growth:2}
		\sup_{\amalg_R} (u-1) \geq (1+\varepsilon)^m > 3.
	\end{equation}
	For this, we need to find a way to apply \zcref{cor:alt_weak_harnack_inequality} appropriately and ensure \((1 + \varepsilon)^m > 3\). Choose an arbitrary integer \(m \geq 1\) such that
	\begin{equation*}
		(1+\varepsilon)^m > 3.
	\end{equation*}
	Now we set \(r = R/(2m)\) and consider the function \(v = u-1\) in
	\begin{equation*}
		\amalg_{2r}^0 := B(z, 2r) \times (R^2-4r^2, R^2) \quad \text{and} \quad \amalg_r^0 := B(z,r) \times (R^2-4r^2, R^2-3r^2).
	\end{equation*}
	We choose \(\eta\) so small that \(\overline{\mu}(E \cap \amalg_r^0) \leq \overline{\mu}(\amalg_r^0)/2\). Indeed, \zcref{prop:balls_ratio_estimate} implies
	\begin{equation*}
		\begin{split}
			\overline{\mu}(E) &\leq \eta \overline{\mu}(\amalg_R) = \eta R^2 V(z,R) \leq \eta R^2 A_1 \left(\frac{R}{r}\right)^{a_1} V(z,r) = \eta \left(\frac{R^2}{r^2}\right) A_1 \left(\frac{R}{r}\right)^{a_1} \overline{\mu}(\amalg_r^0)\\
			&= A_1 \eta (2m)^{2+a_1} \overline{\mu}(\amalg_r^0).
		\end{split}
	\end{equation*}
	We set \(\eta = A_1^{-1} (2m)^{-2-a_1}/2\) to obtain the desired inequality \(\overline{\mu}(E \cap \amalg_r^0) \leq \overline{\mu}(\amalg_r^0) / 2\). Now define
	\begin{equation*}
		H := \{ (x,t) \in \amalg_r^0 : u(x,t) \leq 1\}.
	\end{equation*}
	We immediately obtain  \(\overline{\mu}(H) \geq \overline{\mu}(\amalg_r^0)/2\) and applying \zcref{cor:alt_weak_harnack_inequality} to the function \(u-1\) yields
	\begin{equation*}
		\sup_{\amalg^0_{2r}} (u-1) \geq (1+\varepsilon)(u(z,R^2)-1) \geq 1+\varepsilon.
	\end{equation*}
	Therefore, we find a point \((x_1, t_1) \in \overline{\amalg}_{2r}^0\) with \(u(x_1, t_1) - 1 \geq 1+\varepsilon\). Similarly as before consider the function \(u-1\) in
	\begin{equation*}
		\amalg_{2r}^1 := B(x_1, 2r) \times (t_1 - 4r^2, t_1) \quad \text{and} \quad \amalg_r^1 := B(x_1, r) \times (t_1-4r^2, t_1-3r^2).
	\end{equation*}
	By the same line of arguments as above we have \(\overline{\mu}(E \cap \amalg_r^1) \leq \overline{\mu}(\amalg_r^1)/2\) and obtain according to \zcref{cor:alt_weak_harnack_inequality}
	\begin{equation*}
		\sup_{\amalg_{2r}^1} (u-1) \geq (1+\varepsilon)(u(x_1,t_1) - 1) \geq (1+\varepsilon)^2.
	\end{equation*}
	If we continue this procedure, we obtain points \(\{(x_k, t_k)\}_{k=1}^K\) with
	\begin{equation*}
		u(x_k, t_k) - 1 \geq (1+\varepsilon)^k, \quad d(x_k, x_{k+1}) \leq 2r \quad \text{and} \quad 0\leq t_k - t_{k+1} \leq 4r^2
	\end{equation*}
	for all \( 1 \leq k \leq K\), where \(K\) (possibly infinite) is the index at which the sequence terminates.
	\begin{figure}
		\centering
		\begin{tikzpicture}
  			\begin{scope}
    			\clip (0,-1.125) rectangle (7.5,4.5);
    			\filldraw[fill=blue!30, even odd rule, thick]
      				(-10.85,6.525) circle[radius=14.5]
      				(-10.85,6.525) circle[radius=15];

    			\draw (1.65,0) rectangle (3.9,1.125);
    			\draw (2.2125,0) rectangle (3.3375,0.375);

    			\draw (2.025,1.125) rectangle (4.275,2.25);
    			\draw (2.5875,1.125) rectangle (3.7125,1.5);

    			\draw (2.4375,2.25) rectangle (4.6875,3.375);
    			\draw (3.0,2.25) rectangle (4.125,2.625);

    			\draw (2.625,3.375) rectangle (4.875,4.5);
    			\draw (3.1875,3.375) rectangle (4.3125,3.75);
  			\end{scope}

  			\draw[thick] (0,-1.125) rectangle (7.5,4.5);

  			\fill (3.75,4.5) circle (1pt) node[above, inner sep=1pt] { \((z,R^2)\)};
  			\fill (3.5625,3.375) circle (1pt);
  			\fill (3.15,2.25) circle (1pt);
  			\fill (2.775,1.125) circle (1pt);
  			\fill (2.15,0) circle (1pt);

  			\node[blue!70!black] at (2.025,-0.675) { \(E\)};
  			\draw[->] (4.1,0.5) -- (3,0.1875) node[pos=0, right, inner sep=1pt] { \(\amalg_r^k\)};
  			\node at (3.375,0.75) { \(\amalg_{2r}^k\)};

  			\draw[->] (3,-0.375) -- (2.15,0) node[pos=0, below right, inner sep=1pt] { \((x_{k+1}, t_{k+1})\)};
  			\draw[->] (1.5,1.5) -- (2.775,1.125) node[pos=0, left, inner sep=1pt] { \((x_k, t_k)\)};
  			\node at (6.8, -0.5) { \(\amalg_R\)};
		\end{tikzpicture}
		\caption{Subcylinders \(\amalg_r^k \subset \amalg_{2r}^k \subset \amalg_R\) and \(E\)}
	\end{figure}		
	Since \(4r^2m \leq R^2\) and \(2rm = R\), it follows that \(K \geq m\). Hence \zcref[S]{eq:lem:Lemma_of_Growth:2} holds, and therefore so does \zcref{eq:lem:Lemma_of_Growth:1}.
\end{proof}

\begin{theorem}[Pointwise Harnack Inequality]\label{thm:pointwise_harnack_inequality}
	Let \(u \in C^\infty(\overline{\amalg}_{8R})\) be a non-negative parabolic function. Suppose that
	\begin{equation*}
		\sup_{\amalg_{2R,+}} u = 1.
	\end{equation*}		
	Then there exists a constant \(\gamma = \gamma(A,a,N,p, U) > 0\) such that
	\begin{equation}\label{eq:thm:pointwise_harnack_inequality:1}
		u(z, 64R^2) \geq \gamma.
	\end{equation}
\end{theorem}
\begin{proof}
	Let us define
	\begin{equation*}
		H := \left\{ (x,t) \in \amalg_{4R} : u(x,t) > \frac{1}{2} \right\}.
	\end{equation*}
	If \(\overline{\mu}(H) \geq \delta \overline{\mu}(\amalg_{4R})\), where \(\delta > 0\) will be determined later, then \zcref{thm:weak_harnack_inequality} implies
	\begin{equation*}
		u(z, 64R^2) \geq \frac{1}{2} \varepsilon,
	\end{equation*}
	where \(\varepsilon = \varepsilon(\delta, A, a, N, p, U) > 0\), and \zcref{eq:thm:pointwise_harnack_inequality:1} is proved. Suppose now \(\overline{\mu}(H) < \delta \overline{\mu}(\amalg_{4R})\).\par
	\emph{Step 1:} In this step we construct a sequence of points and cylinders, to which we will apply \zcref{lem:lemma_of_growth} successively. For any point \((x,t) \in \amalg_{8R}\) and any radius \(r > 0\) let us define cylinders
	\begin{equation*}
		\amalg_{r} (x,t) := B(x, r) \times (t-r^2, t)
	\end{equation*}
	and level sets
	\begin{equation*}
		E_{r}^k (x,t) := \left\{ (y,s) \in \amalg_r(x,t) : u(y,s) > 2^{k-1} \right\}
	\end{equation*}
	for any integer \(k \geq 0\).
	Our goal is to construct a sequence of points \((x_k, t_k)\) and radii \(r_k \leq R/2\) such that
	\begin{enumerate}[label=(\alph*)]
		\item \(u(x_k, t_k) \geq 2^k\)
		\item \(\overline{\mu}(E_{r_k}^k (x_k, t_k)) = \eta \overline{\mu}(\amalg_{r_k}(x_k, t_k))\) 
		\item \(\sup_{\amalg_{r_k}(x_k, t_k)} u \geq 2^{k+1}\)
		\item \((x_{k}, t_{k}) \in \overline{\amalg_{r_{k-1}}(x_{k-1}, t_{k-1})}\) for \(k \geq 1\)
	\end{enumerate}
	Here, \(\eta\) is the constant from \zcref{lem:lemma_of_growth} which does not depend on \(k\).
	\begin{figure}
		\centering
		\begin{tikzpicture}
			\filldraw[fill=blue!30, thick] (0,0) .. controls (-3,0.5) and (-3,0) .. (-5,6.5)
  				--
  				(-3.7,6.5) .. controls (-2,1) and (-2,2) .. (0,1.2)
  				-- cycle;
			\begin{scope}
  				\clip (-3.6,1.7) rectangle (-1.4,1.1);
  				\filldraw [fill=orange!40, draw=black, very thin] (-2.3, 1.7) .. controls (-2.3, 1.6) and (-1.9, 1.2) .. (-1.9, 1.1)
    				-- 
    				(-2.3, 1.1) .. controls (-2.3, 1.2) and (-2.7, 1.6) .. (-2.7, 1.7);
			\end{scope}

			\draw[thick] (-6.5,6.5) rectangle (-1.5,3);
			\draw[thick] (-4.8,3) rectangle (-2.5,2.2);
			\draw[thick] (-4,2.2) rectangle (-2,1.7);
			\draw[thick] (-3.6,1.7) rectangle (-1.4,1.1);
			\draw[thick] (-8,6.5) rectangle (0,-0.4);

			\node at (-2.2,6) { \(\amalg_{2R, +}\)};
			\node at (-7.5, 0) { \(\widetilde{\amalg}\)};
			\node[blue!70!black] at (-3.8,5) { \(H\)};

			\fill (-3.65,3) circle (1pt);
			\fill (-3,2.2) circle (1pt);
			\fill (-2.5, 1.7) circle (1pt);
			\fill (-2.1, 1.1) circle (1pt);

			\draw[->] (-2.8, 0.4) -- (-2.1, 1.1) node[pos=0, below, inner sep=1pt] { \((x_{k+1}, t_{k+1})\)};
			\draw[->] (-1.4,2.5) -- (-2.5, 1.7) node[pos=0, above right, inner sep=1pt] { \((x_k, t_k)\)};

			\draw[->, orange!70!black] (-0.8, 1.8) -- (-2.25, 1.2) node[pos=0, right, inner sep=1pt] { \(E_{r_k}^k\)};

			\draw[->] (-4, 1) -- (-3.25, 1.3) node[pos=0, below left, inner sep=1pt] { \(\amalg_{r_k}^k\)};

		\end{tikzpicture}
		\caption{Construction of the sequence of points \(\{(x_k,t_k)\}\)}
	\end{figure}
	At first, we observe that we have for \((x,t) \in \amalg_{2R}\) and \(r = R/2\) according to \zcref{prop:balls_ratio_estimate} the inequality
	\begin{equation*}
		\frac{\overline{\mu}(E_{r}^k(x,t))}{\overline{\mu}(\amalg_r(x,t))} \leq \frac{\overline{\mu}(H)}{\overline{\mu}(\amalg_{R/2}(x,t))} \leq \delta \frac{\overline{\mu}(\amalg_{4R})}{\overline{\mu}(\amalg_{R/2}(x,t))} \leq \delta \cdot 64 A_1 8^{a_1} = \eta
	\end{equation*}
	for \(\delta = \eta /(64A_1 8^{a_1})\). If \(u(x,t) \geq 2^k\) is satisfied, then for \(r\) small enough, the ratio \(\overline{\mu}(E^k_r(x, t)) / \overline{\mu}(\amalg_{r}(x, t))\) is arbitrarily close to \(1\). So we find \(r \leq R/2\) with
	\begin{equation}\label{eq:thm:pointwise_harnack_inequality:2}
		\overline{\mu}(E_{r}^k(x, t)) = \eta \overline{\mu}(\amalg_{r}(x, t)).
	\end{equation}
	Now we are going to construct this sequence. Let \((x_0, t_0)\) be a maximum point of the function \(u\) in the closure of the cylinder \(\amalg_{2R,+}\), that is,
	\begin{equation*}
		\sup_{\amalg_{2R,+}} u = 1 = u(x_0, t_0).
	\end{equation*}
	We find \(r_0 \leq R/2\) such that \zcref{eq:thm:pointwise_harnack_inequality:2} holds. Now we apply \zcref{lem:lemma_of_growth} to the function \(u\) and the sets \(E_{r_0}^0 (x_0, t_0) \subset \amalg_{r_0}(x_0, t_0)\) to obtain
	\begin{equation*}
		\sup_{\amalg_{r_0}(x_0, t_0)} u \geq 2.
	\end{equation*}
	The point \((x_0, t_0)\) satisfies conditions (a) - (d). Now assume we are given a point \((x_k, t_k)\) satisfying (a)-(d). Let \((x_{k+1}, t_{k+1})\) be a maximum point of the function \(u\) in the closure of the cylinder \(\amalg_{r_k}(x_k, t_k)\). That is,
	\begin{equation*}
		u(x_{k+1}, t_{k+1}) = \sup_{\amalg_{r_k}(x_k, t_k)} u \geq 2^{k+1}.
	\end{equation*}
	Let us choose \(r_{k+1}\) such that \zcref{eq:thm:pointwise_harnack_inequality:2} holds for the point \((x_{k+1}, t_{k+1})\) and according to \zcref{lem:lemma_of_growth} we have
	\begin{equation*}
		\sup_{\amalg_{r_{k+1}} (x_{k+1}, t_{k+1})} u \geq 2^{k+2}.
	\end{equation*}
	Altogether, we have inductively constructed a sequence of points \(\{(x_k, t_k)\}_{k \geq 0}\) matching (a)-(d).\par
	\emph{Step 2:} Now we are going to establish the pointwise estimate \zcref{eq:thm:pointwise_harnack_inequality:1}. Let us consider the sequence \(\{(x_k, t_k)\}_{k\geq 0}\) from Step 1. We just consider those \(k\) for which 
	\begin{equation*}
		\amalg_{r_k}(x_k, t_k) \subset B(z, 2R) \times (2R^2,4R^2) =: \widetilde{\amalg}
	\end{equation*}	
	and let \(K\) be the largest index of those \(k\). We have \(K < \infty\) because \(\sup_{\amalg_{8R}} u < \infty\). Since we have \(\amalg_{r_{K+1}}(x_{K+1}, t_{K+1}) \not\subset \widetilde{\amalg}\), it follows
	\begin{equation*}
		\sum_{k=0}^{K+1} r_k \geq R.
	\end{equation*}
	Now we subtract \(r_{K+1} \leq R/2\) to obtain
	\begin{equation*}
		\sum_{k=0}^K r_k \geq \frac{R}{2} > \frac{R}{4} = \frac{R}{4}\sum_{k=0}^\infty \frac{1}{(k+1)(k+2)}.
	\end{equation*}
 	So there must be an integer \(k \leq K\) such that
 	\begin{equation}\label{eq:thm:pointwise_harnack_inequality:3}
 		r_k \geq \frac{R}{4(k+1)(k+2)}.
 	\end{equation}
 	We fix this \(k\) and apply \zcref{lem:iterated_weak_harnack_inequality} to the function \(u/2^{k-1}\) in the cylinders \(\amalg_{r_k}(x_k, t_k) \subset \amalg_{2R}\) and \(\amalg_{8R}\). This yields
 	\begin{equation*}
 		u(z, 64R^2) \geq \const_{A, a ,N, p, U} \left(\frac{\overline{\mu}(\amalg_{r_k}(x_k, t_k))}{\overline{\mu}(\amalg_{4R})}\right)^l 2^{k-1}.
 	\end{equation*}
 	According to \zcref{prop:balls_ratio_estimate} we have the estimate
 	\begin{equation*}
 		\frac{\overline{\mu}(\amalg_{r_k}(x_k, t_k))}{\overline{\mu}(\amalg_{4R})} = \frac{r_k^2}{16R^2} \frac{V(x_k, r_k)}{V(z, 4R)} \geq \frac{1}{A_1} \left(\frac{r_k}{4R}\right)^{a_1 + 2} \geq \const_A ((k+1)(k+2))^{-a_1 - 2},
 	\end{equation*}
 	so
 	\begin{equation*}
 		u(z, 64R^2) \geq \const_{A, a, N, p, U} \frac{2^{k-1}}{((k+1)(k+2))^{l(a_1 + 2})}
 	\end{equation*}
	follows. Note that the preceding inequality only depends on \(A, a, N, p, U\) and \(k\).	If we take the infimum with respect to \(k\), we obtain
	\begin{equation*}
		m := \inf_{k \geq 0} \left\{ \frac{2^{k-1}}{((k+1)(k+2))^{l(a_1 + 2)}} \right\} > 0.
	\end{equation*}
	\(m\) is positive because the numerator grows exponentially, while the denominator grows polynomially. We obtain \zcref{eq:thm:pointwise_harnack_inequality:1} for \(\gamma = \const_{A, a, N, p, U} m\).
\end{proof}

\begin{remark}\label{rem:pointwise_harnack_inequality}
	The idea of the proof of the preceding theorem is attributed to E. M. Landis. A detailed exposition of his argument can be found in \cite{landis1998second}.
\end{remark}

\begin{theorem}[Harnack Inequality]\label{thm:harnack_inequality}
	Let \(u \in C^\infty(\overline{\amalg}_{2R})\) be a non-negative parabolic function in \(\amalg_{2R}\). Then the estimate
	\begin{equation}\label{eq:thm:harnack_inequality:1}
		\sup_{\amalg_{2R,-}} u \leq \const_{A, a, N, p, U} \inf_{\amalg_{2R, +}} u
	\end{equation}
	holds.
\end{theorem}
\begin{proof}
	For \((x,t) \in \amalg_{2R}\) define
	\begin{equation*}
		\amalg_{r}(x,t) := B(x,r) \times (t, t+r^2)
	\end{equation*}
	and the upper  and lower cylinders
	\begin{equation*}
		\amalg_{2r,-} (x,t) := B(x,r) \times (t+r^2, t+2r^2), \ \amalg_{2r,+} (x,t) := B(x,r) \times (t+3r^2, t+4r^2)
	\end{equation*}
	By a similar argument as given in the proof of \zcref{thm:weak_harnack_inequality} it suffices to consider positive functions \(u\). Let us set \(\delta := 8\). The condition \(\amalg_{\delta r}(x,t) \subset \amalg_{2R}\) certainly implies
	\begin{equation}\label{eq:thm:harnack_inequality:2}
		u(x, t+\delta^2r^2) \geq \gamma \sup_{\amalg_{2r, +}(x,t)} u.
	\end{equation}
		\emph{Step 1:} We are going to construct a special sequence to iteratively apply \zcref{eq:thm:harnack_inequality:2}. Let us fix \((x, s) \in \amalg_{2R, -}\) and \((y,t) \in \amalg_{2R, +}\). Our sequence \(\{(x_k, t_k)\}_{k=0}^K\) should satisfy the following properties:
	\begin{enumerate}[label=(\alph*)]
		\item \((x_0, t_0) = (x, s)\) and \((x_K, t_K) = (y, t)\)
		\item \((x_k, t_k) \in \amalg_{2r, +} (x_{k+1}, t_{k+1}-\delta^2r^2)\) and \(\amalg_{\delta r}(x_{k+1}, t_{k+1} - \delta^2 r^2) \subset \amalg_{2R}\).
	\end{enumerate}
	\begin{figure}
		\centering
		\begin{tikzpicture}
			\draw[thick] (0,0) rectangle (9, 6);
			\filldraw	[fill=blue!20, draw=black, thick] (2.25,1.5) rectangle (4.5,4.5);
			\filldraw	[fill=blue!20, draw=black, thick] (6.75,1.5) rectangle (9,4.5);
			\node at (1,4) {\(\amalg_ {2R}\)};
	
			\draw (4.25,3) -- (4.75,3.5) -- (5.2, 2.5);
			\draw[dashed] (5.2, 2.5) -- (5.35, 2.8);
	
			\fill (4.25, 3) circle (1pt) node[left, inner sep=1pt] {\((x,s)\)};	
			\fill (4.75,3.5) circle (1pt) node[above, inner sep=1pt, xshift=9pt] {\((x_1, t_1)\)};	
			\fill (5.2, 2.5) circle (1pt);	

			\fill (5.5,3) circle (0.5pt);
			\fill (5.6,3) circle (0.5pt);
			\fill (5.7,3) circle (0.5pt);
	
			\draw[dashed] (5.85,2.6) -- (6,3);
			\draw (6,3) -- (6.4,2) -- (6.9,3.9) -- (7.3,3);
			\fill (6,3) circle (1pt);
			\fill (6.4,2) circle (1pt);
			\fill (6.9,3.9) circle (1pt) node[above right, inner sep=1pt, xshift=-4pt] {\((x_{K-1}, t_{K-1})\)};
			\fill (7.3,3) circle (1pt)  node[right, inner sep=1pt] { \((y,t)\)};
			\node[blue!70!black] at (8, 1.75) { \(\amalg_{2R, +}\)};
			\node[blue!70!black] at (3.5, 1.75) { \(\amalg_{2R, -}\)};
		\end{tikzpicture}
		\caption{Harnack sequence from \((x,s)\) to \((y,t)\)}
	\end{figure}
	Given \(t_k\), let us pick
	\begin{equation*}
			t_{k+1} = \frac{2\delta^2 - 7}{2}r^2  + t_k
	\end{equation*}
	such that 
	\begin{equation*}
		t_k \in ((t_{k+1} - \delta^2r^2) + 3r^2, (t_{k+1} - \delta^2r^2) + 4r^2)
	\end{equation*}	
	is satisfied. We must ensure
	\begin{equation*}
		\sum_{k=0}^{K-1} (t_{k+1}-t_k) = t - s,
	\end{equation*}				
	that is, 
	\begin{equation*}
		\frac{2\delta^2 - 7}{2} K r^2 = t-s.
	\end{equation*}				
	Rearranging terms leads to the choice
	\begin{equation}\label{eq:thm:harnack_inequality:3}
		r = \sqrt{\frac{2(t-s)}{K(2\delta^2 - 7)}} \simeq \frac{R}{\sqrt{K}},
	\end{equation}
	where we have used the relation \(R^2 \leq t-s \leq 4R^2\). To connect \(x, y \in B(z, R)\), the condition
	\begin{equation*}
		K r \geq 2R,
	\end{equation*}				
	must hold because \(d(x_k, x_{k+1}) < r\). We can arrange this by noting
	\begin{equation*}
		4R \leq \const 2\sqrt{K}r \leq 2Kr
	\end{equation*}
	for \(K\) large enough. It remains to ensure
	\begin{equation*}
		\amalg_{\delta r} (x_{k+1}, t_{k+1} - \delta^2r^2) \subset \amalg_{2R}.
	\end{equation*}
	We choose the points \(x_k\) along a path from \(x\) to \(y\) through \(z\) such that \(x_k \in B(z,R)\). Hence \(B(x_{k+1}, \delta r) \subset B(z,2R)\) provided \(\delta r < R\). Moreover,
    \begin{equation*}
        t_{k+1} - \delta^2 r^2 > R^2 - \delta^2 r^2 > 0
    \end{equation*}
    if \(\delta r < R\), while \(t_{k+1} < 4R^2\), hence it  suffices to ensure \(\delta < R/r\). We have according to \zcref{eq:thm:harnack_inequality:3}
	\begin{equation*}
		\frac{R}{r}  \geq \const \sqrt{K} > \delta
	\end{equation*}
	for \(K\) large enough.\par
	\emph{Step 2:} Fix two arbitrary points \((x, s) \in \amalg_{2R,-}\) and \((y, t) \in \amalg_{2R, +}\). We claim
	\begin{equation}\label{eq:thm:harnack_inequality:4}
		u(x, s) \leq \const_{A, a, N, p, U} u(y, t).
	\end{equation}
	\zcref[S]{eq:thm:harnack_inequality:4} clearly implies \zcref{eq:thm:harnack_inequality:1}. Let \(\{(x_k, t_k)\}_{k=0}^K\) be the sequence constructed in Step 1 for the values \((x, s)\) and \((y, t)\). We have
	\begin{equation*}
		u(x_{k+1}, t_{k+1}) \geq \gamma u(x_k, t_k)
	\end{equation*}
	for all \(k \leq K-1\) due to \zcref{eq:thm:harnack_inequality:2}. Iterating this \(K\) times leads to
	\begin{equation*}
		u(y, t) \geq \gamma^K u(x, s),
	\end{equation*}
	and \zcref{eq:thm:harnack_inequality:4} follows since \(K\) does not depend on \((x,s)\) and \((y,t)\).
\end{proof}

We now only assume that \((M,g,\mu)\) is a smooth, non-compact, connected, and complete weighted Riemannian manifold, possibly with boundary.

\begin{theorem}\label{thm:harnack_inequality_equivalent}
    The parabolic Harnack inequality for the heat equation holds if and only if it holds for every parabolic differential operator \(L\) in the class considered above.
\end{theorem}
\begin{proof}
    Assume first that the parabolic Harnack inequality holds for the heat equation. Then the parabolic Harnack inequality implies two-sided Gaussian heat kernel bounds; see \cite{grigoryan1991heat,saloff1992note,saloff2002aspects, fabesstroock1986harnack,grigoryan1994heat,kohlmeier2025gaussian}. These bounds imply the volume doubling property \ref{enum:introduction:1} and the Poincar\'e inequality \ref{enum:introduction:2}; see \cite{grigoryan1991heat,saloff1992note,saloff2002aspects,
    kohlmeier2025gaussian}. Hence \zcref{thm:harnack_inequality} yields the parabolic Harnack inequality for every operator \(L\) in the class considered above.

    Conversely, choosing \(p\equiv1\) and \(U(x,t)=\mathrm{Id}_{T_xM}\) gives the heat operator, so the reverse implication is immediate.
\end{proof}

\printbibliography

\bigskip

\begin{small}
	\noindent
	Stefan Christian Kohlmeier\\
	Institute of Mathematics, Paderborn University\\
	\textit{E-mail:} {\hypersetup{allcolors=black}\href{mailto:stekoh@mail.uni-paderborn.de}{\texttt{stekoh@mail.uni-paderborn.de}}}
\end{small}

\end{document}